\documentclass[11pt]{amsart}
\usepackage[T1]{fontenc}
\usepackage{txfonts}
\usepackage{amssymb,verbatim}
\usepackage[all]{xy}
\usepackage{subfig}
\usepackage{tikz}
\usepackage{color}
\usepackage{a4wide}
\usepackage{mathrsfs}
\usepackage{hyperref}
\usepackage{wrapfig}
\usetikzlibrary{arrows}
\DeclareMathAlphabet{\mathpzc}{OT1}{pzc}{m}{it}

\newcommand{\R}{\mathbb{R}}
\newcommand{\C}{\mathbb{C}}

\newcommand\Z{\mathbb{Z}}
\newcommand{\N}{\mathbb{N}}
\newcommand{\Q}{\mathbb{Q}}
\renewcommand{\S}{\mathbb{S}}

\newcommand{\Lb}{\mathbf{L}}

\newcommand{\Pb}{\mathbb{P}}

\newcommand{\Tbb}{\mathbb{T}}

\newcommand{\Ub}{\mathbf{U}}

\newcommand{\pp}{\mathbf{p}}

\newcommand{\xx}{\mathbf{x}}

\newcommand{\Ccal}{\mathcal{C}}

\newcommand{\Fcal}{\mathcal{F}}

\newcommand{\Kcal}{\mathcal{K}}
\newcommand{\Lcal}{\mathcal{L}}
\newcommand{\Mcal}{\mathcal{M}}

\newcommand{\Ocal}{\mathcal{O}}

\newcommand{\Tcal}{\mathcal{T}}
\newcommand{\Ucal}{\mathcal{U}}

\newcommand{\QQ}{\mathscr{Q}}
\newcommand{\TT}{\mathscr{T}}

\newcommand{\id}{\mathrm{id}}

\newcommand{\SL}{{\rm SL}}
\newcommand{\GL}{{\rm GL}}

\newcommand{\inter}{\mathrm{int}}

\newcommand{\vol}{{\rm vol}}

\newcommand{\proj}{\mathrm{pr}}
\newcommand{\Ur}{\mathrm{U}}

\newcommand{\ST}{\mathcal{ST}}

\newcommand{\pr}{\mathrm{pr}}

\newcommand{\hg}{\hat{g}}
\newcommand{\hn}{\hat{n}}
\newcommand{\hx}{\hat{x}}
\newcommand{\hX}{\hat{X}}

\newcommand{\hZ}{\hat{Z}}

\newcommand{\homg}{\hat{\omega}}
\newcommand{\hf}{\hat{f}}
\newcommand{\hk}{\hat{k}}
\newcommand{\hxx}{\hat{\mathbf{x}}}

\newcommand{\Sig}{\Sigma}

\newcommand{\G}{\Gamma}

\newcommand{\omg}{\omega}
\newcommand{\ol}{\overline}
\newcommand{\ul}{\underline}
\newcommand{\lra}{\longrightarrow}
\newcommand{\ra}{\rightarrow}

\newcommand{\cH}{\mathcal{H}}
\newcommand{\cM}{\mathcal{M}}
\newcommand{\cS}{\mathcal{S}}
\newcommand{\cU}{\mathcal{U}}

\newcommand{\HH}{\mathrm{H}}
\newcommand{\sbt}{\bullet}
\newcommand{\cW}{\mathcal{W}}
\newcommand{\cL}{\mathcal{L}}
\newcommand{\cV}{\mathcal{V}}

\newcommand{\Skdiff}{\Omega^k\mathcal{M}_{g}(\ul{k})}
\newcommand{\Skdiffi}{\Omega^k_1\mathcal{M}_{g}(\ul{k})}
\newcommand{\Sabel}{\Omega\Mcal_{\hat{g}}(\ul{\hat{k}})}
\newcommand{\pSkdiff}{\Pb\Omega^k\mathcal{M}_{g}(\ul{k})}

\newcommand{\Tkdiff}{\Omega^k\Mcal^\vartriangle_g(\ul{k})}
\newcommand{\Qkdiff}{\Omega^k\Mcal^\square_g(\ul{k})}
\newcommand{\Tkdiffm}{\Omega^k\Mcal^\vartriangle_g(\ul{k},m)}
\newcommand{\Qkdiffm}{\Omega^k\Mcal^\square_g(\ul{k},m)}

\newcommand{\trate}{\Omega\Mcal_{g}(\ul{k})}

\newcommand{\htrate}{\Omega\Mcal_{g}}

\newtheorem{Theorem}{Theorem}[section]

\newtheorem{Lemma}[Theorem]{Lemma}
\newtheorem{Proposition}[Theorem]{Proposition}

\newtheorem{Definition}[Theorem]{Definition}

\newtheorem{Claim}[Theorem]{Claim}

\theoremstyle{remark}
\newtheorem{Remark}[Theorem]{Remark}
\begin{document}
\title[VHS and enumerating tilings]{Variation of Hodge structure and enumerating tilings of surfaces by triangles and squares}


\author{Vincent Koziarz}
\address{Univ. Bordeaux, IMB, CNRS, UMR 5251, F-33400 Talence, France}
\email[V.~Koziarz]{vincent.koziarz@math.u-bordeaux.fr}

\author{Duc-Manh Nguyen}
\address{Univ. Bordeaux, IMB, CNRS, UMR 5251, F-33400 Talence, France}
\email[D.-M.~Nguyen]{duc-manh.nguyen@math.u-bordeaux.fr}

\date{\today}

\begin{abstract}
Let $S$ be a connected closed oriented surface of genus $g$.
Given a triangulation (resp. quadrangulation) of $S$, define the index of each of its vertices to be the number of edges originating from this vertex minus $6$ (resp. minus $4$). Call the set of integers recording the non-zero indices the profile of the triangulation (resp. quadrangulation).
If $\kappa$ is a profile for triangulations (resp. quadrangulations) of $S$, for any $m\in \Z_{>0}$, denote by $\TT(\kappa,m)$ (resp. $\QQ(\kappa,m)$) the set of (equivalence classes of) triangulations (resp. quadrangulations) with profile $\kappa$ which contain at most $m$ triangles (resp. squares).
In this paper, we will show that if $\kappa$ is a profile for triangulations (resp. for quadrangulations) of $S$ such that none of the indices in $\kappa$ is divisible by $6$ (resp. by $4$), then $\TT(\kappa,m)\sim c_3(\kappa)m^{2g+|\kappa|-2}$ (resp. $\QQ(\kappa,m)\sim c_4(\kappa)m^{2g+|\kappa|-2}$), where $c_3(\kappa) \in \Q\cdot(\sqrt{3}\pi)^{2g+|\kappa|-2}$ and $c_4(\kappa)\in \Q\cdot\pi^{2g+|\kappa|-2}$.
The key ingredient of the proof is a result of J.~Koll\'ar~\cite{Kollar87} on the link between the curvature of the Hogde metric on vector subbundles of a variation of Hodge structure over algebraic varieties, and Chern classes of their extensions.
By the same method, we also obtain the rationality (up to some power of $\pi$) of the Masur-Veech volume of arithmetic affine submanifolds of translation surfaces that are transverse to the kernel foliation.
\end{abstract}

\maketitle

%

\section{Introduction}\label{sec:intro}
\subsection{Triangulations and quadrangulations of surfaces}\label{sec:triang:quad}
Let $S$ be a connected closed oriented surface of genus $g \geq 0$. A {\em triangulation} (resp. {\em quadrangulation}) of $S$ is an embedded graph $\G$ on $S$ such that each component of the complement of $\G$ is homeomorphic to a disc and bounded by $3$ edges (resp. $4$ edges). A component of $S\setminus \G$ is called a {\em face} of the triangulation (resp. quadrangulation).  Note that an edge of $\G$ can appear twice in the boundary of the same face. Two triangulations (resp. quadrangulations) of $S$ are said to be {\em equivalent} if there is a homeomorphism of $S$ which restricts to an isomorphism  between the corresponding embedded graphs.


The valency $e_v$ of a vertex $v$ of $\G$  is the number of directed edges originating from $v$. Note that a loop at $v$ counts twice in $e_v$. If $\G$ is a triangulation, define the index of $v$ to be $\kappa(v):=e_v-6$. If $\G$ is a quadrangulation then the index of $v$ is $\kappa(v):=e_v-4$. The vertices whose index is not zero are said to be {\em singular}. By computing the Euler characteristic of $S$, one readily finds
\begin{equation}\label{eq:euler:char}
\sum_{v \, {\rm singular}}\kappa(v)= \left\{
\begin{array}{ll}
  12(g-1), & \hbox{if $\G$ is a triangulation}, \\
  8(g-1), & \hbox{if $\G$ is a quadrangulation}.
\end{array}
\right.
\end{equation}
We will call the sequence of numbers $\bigl(\kappa(v), \, v \textrm{ singular vertex}\bigr)$ the profile of $\G$.
Let  $\kappa=(\kappa_1,\dots,\kappa_n)$ be a sequence of integers.
We say that $\kappa$ is an {\em admissible profile for triangulations}  of $S$ if
\begin{itemize}
 \item[$\bullet$] $\kappa_i>-6$ and $\kappa_i\neq 0$, for all $i=1,\dots,n$,
  \item[$\bullet$] $\kappa_1+\dots+\kappa_n=12(g-1)$.
\end{itemize}
Similarly,  we will say that $\kappa$ is {\em an admissible profile for quadrangulations}  of $S$ if
\begin{itemize}
 \item[$\bullet$] $\kappa_i>-4$ and $\kappa_i\neq 0$, for all $i=1,\dots,n$,
  \item[$\bullet$] $\kappa_1+\dots+\kappa_n=8(g-1)$.
\end{itemize}


\medskip

Given an admissible profile $\kappa$ for triangulations of $S$, for any $m\in \Z_{>0}$, we denote by $\TT(\kappa,m)$ the set of equivalence classes of triangulations of $S$ with profile $\kappa$ and  number of faces at most $m$. In the same manner, if $\kappa$ is an admissible profile for quadrangulations of $S$, we denote by $\QQ(\kappa,m)$ the set of equivalence classes of quadrangulations of $S$ with profile $\kappa$ and number of faces at most $m$.
In this paper, we will show
\begin{Theorem}\label{th:asymp:count}\hfill
\begin{itemize}
\item[(i)] Let $\kappa=(\kappa_1,\dots,\kappa_n)$ be an admissible profile for triangulations of $S$. If $(\kappa_1,\dots,\kappa_n)$ satisfies   $\kappa_i\not\in 6\cdot\Z$ for all $i=1,\dots,n$,  then we have
    \begin{equation}\label{eq:asymp:triang}
    \lim_{m \ra \infty} \frac{\#\TT(\kappa,m)}{m^{2g+n-2}}=c_3(\kappa)
    \end{equation}
    where $c_3(\kappa)$ is a constant  in $\Q\cdot(\sqrt{3}\pi)^{2g+n-2}$.

\item[(ii)] Let $\kappa=(\kappa_1,\dots,\kappa_n)$ be an admissible profile for quadrangulations of $S$. If $(\kappa_1,\dots,\kappa_n)$ satisfies  $\kappa_i\not\in 4\cdot\Z$  for all $i=1,\dots,n$, then we have
    \begin{equation}\label{eq:asymp:quad}
    \lim_{m \ra \infty} \frac{\#\QQ(\kappa,m)}{m^{2g+n-2}}=c_4(\kappa)
    \end{equation}
    where $c_4(\kappa)$ is a constant in $\Q\cdot\pi^{2g+n-2}$.
\end{itemize}
\end{Theorem}

\begin{Remark}\label{rk:limit:known}\hfill
\begin{itemize}
\item 
In \cite{Thurston98}, Thurston studied triangulations of the sphere where the valency of every vertex is at most $6$.
He relates the asymptotics of the number of such triangulations with the volume of the moduli space of pointed genus zero curves with respect to some complex hyperbolic metric. Those volumes have been computed by different methods in \cite{McMullen17} and \cite{KN18}.
The problem of enumerating tilings of surfaces by triangles and squares has also been addressed in \cite{ES18,Engel-I,Engel-II}.

\item The existence of the limits in \eqref{eq:asymp:triang} and \eqref{eq:asymp:quad} is a consequence of the main theorem of \cite{Ngu19}. In \cite{Engel-I}, Engel shows that the limits in \eqref{eq:asymp:triang} and  \eqref{eq:asymp:quad}, if finite, must belong to the ring $K[\pi]$, where $K$ is either $\Q$ or $\Q(\sqrt{3})$ ($K=\Q$ for quadrangulations). The main content of Theorem~\ref{th:asymp:count} is that these constants belong to $\Q\pi^{2g+n-2}$ or to $\Q(\sqrt{3}\pi)^{2g+n-2}$ in the case $\kappa$ satisfies the hypothesis of  (i) and (ii).
%
%
\end{itemize}
\end{Remark}

\subsection{Enumerating square-tiled surfaces in affine invariant submanifolds}\label{sec:enum:sq:tiled}
Translation surfaces are pairs $(X,\omega)$ where $X$ is a compact Riemann surface and $\omega$ a non-zero holomorphic $1$-form on $X$. The $1$-form $\omg$ defines a flat metric with conical singularities at its zeros.
A square-tiled surface is a pair $(X,\omega)$ where $\omega$ is the pull-back of the $1$-form $dz$ on the standard torus $\Tbb=\C/(\Z\oplus\Z\imath)$ via a ramified cover $f: X\ra \Tbb$, which is branched over a unique point.


The space of translation surfaces of genus $g\geq 2$ is naturally stratified by the orders of the zeros of $\omega$. Given an $n$-tuple of positive integers $\ul{k}=(k_1,\dots,k_n)$ such that $k_1+\dots+k_n=2g-2$, denote by $\trate$ the set of translation surfaces $(X,\omega)$ such that $\omega$ has exactly $n$ zeros with orders  given by $(k_1,\dots,k_n)$.
It is well known  that $\trate$ is a complex orbifold of dimension $2g+n-1$. For any $(X,\omega)\in \trate$, a neighborhood of $(X,\omega)$ can be identified with an open subset of $H^1(X,\{\textrm{zeros of} \; \omega\}; \C)$ by local charts called {\em period mappings}.

There is an action of $\GL^+(2,\R)$  on $\trate$ defined as follows: let $(z_1,\dots,z_d)$ be some local coordinates of $\trate$ given by a period mapping, and $A$ a matrix in  $\GL^+(2,\R)$. Then the action of $A$ is given by $A: (z_1,\dots,z_d) \mapsto (A(z_1),\dots,A(z_d))$, where  $A$ acts on $\C$ via the standard identification $\C\simeq \R^2$. The dynamics of this action of $\GL^+(2,\R)$ has deep connections with various domains such as billiards in rational polygons, interval exchange transformations, Teichm\"uller dynamics in moduli space (see for instance \cite{MaTa02, Zorich:survey}).

The properties of the $\GL^+(2,\R)$ action, in particular the structure of the orbit closures, are the subject of a fast growing literature in the last few decades. It follows from the groundbreaking results of Eskin-Mirzakhani \cite{EM18} and Eskin-Mirzakhani-Mohammadi~\cite{EMM15} that  any $\GL^+(2,\R)$-orbit closure is  an  immersed suborbifold of $\trate$, which is locally defined by linear equations with real coefficients in local charts by period mappings. Such suborbifolds are commonly known as {\em invariant affine submanifolds} (or {\em affine submanifolds} for short) of $\trate$.

An affine submanifold $\cM$ is said to be {\em arithmetic} if it is locally defined by linear equations with coefficients in $\Q$. It is shown in \cite{Wright14} that $\Mcal$ is arithmetic if and only if it contains a square-tiled surface.
Our second main result concerns the enumerating of square-tiled surfaces in arithmetic affine submanifolds.
Before giving  the statement of the second theorem, let us recall some relevant features of affine submanifolds.
Each stratum of translation surfaces carries naturally two local systems $\cH^1_{\rm rel}$ and $\cH^1$ whose fibers over $(X,\omega)$ are respectively $H^1(X,\{\textrm{zeros of } \omega\};\C)$ and $H^1(X,\C)$. Let $\pp: H^1(X,\{\textrm{zeros of } \omega\};\C) \ra H^1(X,\C)$ denote the natural projection. Then $\pp$ gives rise to a morphism $\pp: \cH^1_{\rm rel} \ra \cH^1$ of local systems over $\trate$.
By definition, the tangent bundle $T\cM$ of an affine submanifold $\cM$ is a local subsystem of $\cH^1_{\rm rel}$ over $\cM$.
Moreover, we have a foliation of $\cM$, called the {\em kernel foliation}, whose tangent space at every point is identified with $\ker(\pp)$.
We will say that $\Mcal$ is {\em absolutely rigid} if the restriction of $\pp$ to $T\Mcal$ is injective.
Equivalently, $\Mcal$ is absolutely rigid if it is transverse to the {\em kernel foliation} of $\trate$.
Examples of such affine submanifolds include strata of translation surfaces having a single singularity (minimal strata), double covers of quadratic differentials which only have zeros of odd order, and closed orbits generated by square-tiled surfaces (this list is not exhaustive).

\begin{Theorem}\label{th:aff:subman:count}
Let $\Mcal$ be an {\bf arithmetic absolutely rigid} affine submanifold of $\trate$ of dimension $d$. For $m\in \Z_{>0}$, let $\ST(\Mcal,m)$ denote the set of square-tiled surfaces in $\Mcal$ that are formed by at most $m$ unit squares. We then have
\begin{equation}\label{eq:asymp:sq:tiled}
\lim_{m \ra \infty}\frac{\#\ST(\Mcal,m)}{m^{d}}=c(\Mcal)
\end{equation}
where $c(\Mcal) \in \Q\pi^{d}$.
\end{Theorem}
The limit on the left hand side of \eqref{eq:asymp:sq:tiled} is often referred to  as the Masur-Veech volume of $\cM$.
The rationality of this volume has been known for several classes of arithmetic absolutely rigid affine submanifolds.
If $\dim \cM=2$, then $\cM$ is the $\GL^+(2,\R)$-orbit of a square-tiled surface. In this case, it is a well known fact that the projection of $\cM$ in the moduli space of Riemann surfaces is a Teichm\"uller curve (see for instance \cite{SW04}). Up to a universal constant in $\Q\cdot\pi^2$, the Masur-Veech volume of $\cM$ is equal to the Euler characteristic of this Teichm\"uller curve.
In the case $\cM$ is a minimal stratum (which consists of Abelian differentials with  a single zero), the rationality of the Masur-Veech volume was proven in the work~\cite{EO01} (see also \cite{Sau18} for related formulas). In the case $\cM$ consists of double covers of quadratic differentials with odd order zeros and simple poles, this rationality  was shown in \cite{AEZ16} for genus zero case, and in \cite{CMS19} for the general case (see also \cite{EO06,Goujard}).
An unexpected arithmetic absolutely rigid affine submanifold of dimension 4 in genus four was discovered in \cite{MMW17}.
Its Masur-Veech volume was computed in \cite{TT19}.
However, to the authors' knowledge, for general  arithmetic absolutely rigid affine submanifolds, the rationality of the Masur-Veech volume has not been known.

%
%

\subsection{Outline  and remarks on the proof of the main theorems}\label{subsec:outline}
The proofs of Theorem~\ref{th:asymp:count} and Theorem~\ref{th:aff:subman:count} go as follows: we first relate the asymptotics we are interested in to the Masur-Veech volumes of some moduli spaces of projectivized pluri\-differentials on Riemann surfaces. The moduli spaces under consideration belong to a special class of subvarieties, which will be called {\em linear submanifolds}, of the projectivizations of strata of Abelian differentials (cf. Definition~\ref{def:lin:man}). By construction, those moduli spaces carry a tautological line bundle, which comes equipped with a natural Hermitian metric.

Under some appropriate hypothesis, that is the linear submanifolds are supposed to be polarized and absolutely rigid (cf. \textsection\ref{subsec:linsubrev}), we then show that up to a rational constant, the Masur-Veech volume form is pointwise equal to some power of the curvature form of the natural metric on the tautological line bundle.
To show the rationality of the Masur-Veech volumes (up to multiplication by some power of $\pi$), instead of constructing specific compactifications for the linear submanifolds,  we will make use of the variation of (mixed) Hodge structure over those submanifolds.

To fix ideas, let us denote by $\cM$ a linear submanifold of a stratum of Abelian differentials, and by $\Pb\cM$ its projectivization.
By definition, there is a variation of $\Z$-mixed Hodge structure over $\Pb\cM$.
The tautological line bundle is actually a holomorphic line subbundle of the vector bundle associated with the $\Z$-local system of this variation of Hodge structure (VHS).
Up to taking some finite cover, and some modification of an arbitrary compactification of $\Pb\cM$ with normal crossing boundary, one can show that the tautological line bundle extends as a line subbundle of the canonical extension of the VHS.
Since the Hermitian metric on the tautological line bundle coincides with the Hodge metric of the VHS, it follows from a result of J. Koll\'ar \cite{Kollar87} that any power of the curvature of this metric is a representative in the sense of currents of the corresponding power of the first Chern class of the extended line bundle.  Since the Masur-Veech volume of $\Pb\cM$ is equal to the integral of the maximal power of this curvature form multiplied by a rational number, this allows us to conclude.

A few comments on the strategy of the proof are in order. First, the relation between the Masur-Veech volumes and the asymptotics of the counting problems has been known since the work~\cite{EO01}. Second, that the Masur-Veech volume form is proportional to the top power of the curvature of the tautological line bundle (on the associated projectivized moduli spaces) has been known to experts in the field.
For minimal strata of Abelian differentials and quadratic differentials with odd order zeros, their ratios were computed in \cite{Sau18, CMS19}.
However, to the authors' knowledge, for moduli spaces of $k$-differentials (where $k\in \{2,3,4,6\}$) this ratio has not been explicitly calculated  in the literature.
In this paper, we limit ourselves to showing that this ratio is a rational number in these cases (cf. Proposition~\ref{prop:compare:meas:arith} and Proposition~\ref{prop:ratio:vol:forms}).

Finally, for Theorem~\ref{th:asymp:count} an alternate method to show that the integral of the top power of the curvature form gives a rational number is to use the compactifications of the corresponding moduli spaces of $k$-differentials constructed in \cite{BCGGM:multiscale, CMZ19}. Indeed, the results in \cite{CMZ19} imply that the integral under consideration is equal to the integral of some power of the first Chern class of a line bundle over a compact complex orbifold, thus must be a rational number. In \cite{CMS19}, this method has been used to calculate the Masur-Veech volumes of some moduli spaces of quadratic differentials.
The main novelty of the current paper consists in the use of  variation of Hodge structure and Koll\'ar's result, which  bypasses the involving construction of the compactifications in \cite{BCGGM:multiscale, CMZ19}, and provides a uniform treatment for  strata of $k$-differentials (Theorem~\ref{th:asymp:count}) and absolutely rigid linear submanifolds (Theorem~\ref{th:aff:subman:count}).
The drawback is that our approach does not provide an effective way to compute the corresponding Masur-Veech volumes.
Nevertheless, one may expect that in some situations, where the construction of the finite cover mentioned above can be  carried out explicitly,  it is possible to obtain computable formulas for the Masur-Veech volumes by this method.

%

\subsection{Organization}\label{subsec:organization}
The paper is organized as follows: in \textsection\ref{sec:strate:n:lin:subvar} we recall basic properties of strata of Abelian differentials and introduce the notion of linear submanifolds as well as the variation of mixed Hodge structure over  these varieties. In \textsection\ref{sec:volform}, we give the definition of several natural volume forms on linear submanifolds and the relations between them. In \textsection\ref{sec:rational:vol}, we prove the rationality of the volumes of the projectivized linear submanifolds  that are polarized and absolutely rigid (cf. Theorem~\ref{th:vol:rational}). The proofs of Theorem~\ref{th:aff:subman:count} and of Theorem~\ref{th:asymp:count} are given in \textsection\ref{sec:proof:theorem2} and \textsection\ref{sec:count:tiling} respectively.

\subsection*{Acknowledgements.} We are grateful to Yohan Brunebarbe for explaining to us the ingredients of the proof of Theorem~\ref{th:vol:rational}, and for sharing with us his point of view on the construction of the volume form. 

\section{Moduli spaces of Abelian differential and linear subvarieties}\label{sec:strate:n:lin:subvar}
\subsection{Moduli spaces of Abelian differential}\label{subsec:strate:abel}
The moduli space $\htrate$ of pairs $(C, \omega)$ where $C$ is a smooth complex curve of genus $g$ and $\omega$ is a non trivial Abelian differential (i.e. a holomorphic 1-form) is the total space of the Hodge bundle over the moduli space $\cM_g$ of smooth curves of genus $g$, with the zero section removed.

The space $\htrate$  is an orbifold which is naturally stratified by the multiplicities of zeroes of the corresponding Abelian differentials. For any partition $\underline k=(k_1,\dots,k_n)$ of $ 2g - 2$ by positive integers, the associated stratum is a locally closed subset of $\htrate$ for the Zariski topology. The strata are always non-empty but not necessarily connected, though each has no more than three connected components \cite{KZ03}. Each stratum is a complex algebraic variety with a complex orbifold structure that will be denoted by $\trate$. To lighten the notation, throughout this paper, we use $\trate$ to designate a connected component of the corresponding stratum.

If $(C,\omega)\in\trate$, we will use the notation $Z(\omega)=\{x_1,\dots,x_n\}$, where the point $x_i\in C$ is a zero of $\omega$ of order $k_i$. We have a preferred atlas on $C - Z(\omega)$ given by the local primitives of the closed form $\omega$. Two charts in this atlas differ by a translation, so that this atlas defines on $C - Z(\omega)$ a flat metric structure with cone singularities at $Z(\omega)$. Its area is given by $A(C, \omega) = \frac{\imath}{2} \int_C \omega \wedge \bar \omega$.

In some situations, it is relevant to consider Abelian differentials $(C,\omg)$ together with some marked points on $C$ that are not zeros of $\omg$. Those marked points can be considered as zeros of order $0$ of $\omg$. Let $\ul{k}=(k_1,\dots,k_n)$ be a vector of {\bf non-negative} integers such that $k_1+\dots+k_n=2g-2$. We denote by $\trate$   the space of triples $(C,\omg,Z)$, where $(C,\omg)\in \htrate$, and $Z=\{x_1,\dots,x_n\}$ is a finite subset of $C$ such that $\mathrm{div}(\omg)=k_1x_1+\dots+k_nx_n$. Note that we do not fix any preferred numbering on the points in $Z$, the only requirement is that $x_i$ is a zero of order $k_i$.
Since the elements of $\htrate$ do not record the location of the marked points,  in the case some of the $k_i$'s are $0$, $\trate$ is not a subvariety of $\htrate$.
Nevertheless, it is well known that $\trate$  still enjoys the same properties as the strata of $\htrate$, in particular, $\trate$ is an algebraic  variety and has an orbifold structure as a complex analytic space.
In what follows we will call $\trate$ a  stratum of Abelian differentials indifferently whether $\ul{k}$ has $0$ entries or not.

\subsection{Period coordinates and linear submanifolds}\label{sec:local:coord}
Since each stratum is an orbifold, a local chart on an open subset $\cU\subset\trate$ will have to be understood as a chart over a finite (ramified) covering $\hat\cU$ of $\cU$, where $\hat\cU$ can be chosen simply connected, endowed with a linear action of a finite group $\Gamma_\cU$ such that $\cU=\hat\cU/\Gamma_\cU$. Objects over $\trate$ will be defined locally over $\hat\cU$ and endowed with an action of $\Gamma_\cU$.  Alternately, one can define $\trate$ as a Deligne-Mumford stack, but we will not use this point of view.

We fix a stratum $\cS=\trate$ and let $\cU=\hat\cU/\Gamma_\cU$ be a neighborhood of $(C,\omega)$, where $\hat\cU$ is as above. Then, for all $(C',\omega')\in\hat\cU$ one can canonically identify the relative homology group $\HH_1(C', Z(\omega'); \Z)$  and its dual $\HH^1(C', Z(\omega'); \Z)$ with $H_1(C,Z(\omega);\Z)$ and $H^1(C,Z(\omega);\Z)$ respectively.
Integrating the form $\omega$ along relative cycles, we obtain a class in $\HH^1(C,Z(\omega); \C)=\HH_{1}(C,Z(\omega); \C)^\vee$. A fundamental result of Veech \cite{Veech90} asserts that the resulting map $\hat\cU\rightarrow \HH^1 (C, Z(\omega); \C)$
(usually called a period mapping) is a local biholomorphism. In this way, we define a linear structure on $\trate$, meaning that we get an atlas with linear changes of coordinates. A choice of a $\Z$-basis of $\HH_1 (C, Z(\omega); \Z)$ will provide us with local coordinates around $(C, \omega)$ in $\trate$ that will be called period coordinates. Note that changes of period coordinates are given by matrices with integral coefficients.

There is a natural $\C^*$-action on $\trate$ by multiplying the Abelian differential by a scalar.
We will denote by $\Pb\trate=\trate/\C^*$ the projectivization of  $\trate$.
If $(C,\omg)$ is an element of $\trate$, its projectivization in $\Pb\trate$ will be denoted by $(C,[\omg])$.
By definition, $\Pb\trate$ is a locally closed subset of $\Pb\Omega\cM_g$, and $\trate$ can be interpreted as the total space of the tautological line bundle over $\Pb\trate$ with the zero section removed.

For our purpose, we will be particularly interested in the following class of subvarieties of $\trate$.

\begin{Definition}\label{def:lin:man}
A linear submanifold $\cM$ of $\trate$ is a complex algebraic subvariety such that the local irreducible components  of $\cM$ are defined by linear equations in period coordinates.
If $\cM$ is a linear submanifold of $\trate$, by a slight abuse of language, we will call $\Pb\cM:=\cM/\C^*$ a linear submanifold of $\Pb\trate$.
\end{Definition}

\begin{Remark}\label{rk:lin:subman:orbifold}
By definition, every local branch of $\cM$ (considered as a complex analytic space) corresponds to a vector subspace in local charts by period mappings of $\trate$.
Recall that  $\trate$ has a structure of a complex orbifold. It follows that the normalization of $\cM$ (where all the local branches are separated) also has a structure of complex orbifold.
\end{Remark}

Originally linear submanifolds arose from the study of $\GL^+(2,\R)$-action on $\trate$. It follows from the works  \cite{EM18, EMM15, Filip16} that every $\GL^+(2,\R)$-orbit closure in $\trate$ is a linear submanifold locally defined by linear equations with real coefficients. These are commonly known as {\em invariant affine submanifolds}. Other samples of linear subvarieties arise from strata of moduli spaces of pluri\-differentials (cf. \textsection\ref{sec:triang:quad}). In this case, locally the subvarieties are defined by linear equations with coefficients in $\Q(\zeta)$, where $\zeta$ is root of unity. In this paper, we are essentially concerned with these two families of linear submanifolds.
Note also that in \cite[Def. 6.4]{Moller08}, M\"oller introduced a notion of linear manifold which is similar to ours, but somewhat more restrictive.

\subsection{Numbered zeros and marked points}\label{subsec:numb:zeros}
Let $\cM_{g,n}$ denote the moduli space of $n$-pointed genus $g$ smooth curves. Given a vector $\ul{k}=(k_1,\dots,k_n)$ of non-negative integers such that $k_1+\dots+k_n=2g-2$, we denote by $\cM_{g,n}(\ul{k})$ the set of $(C,x_1,\dots,x_n) \in \cM_{g,n}$ such that $k_1x_1+\dots+k_nx_n$ is the zero divisor of a holomorphic $1$-form on $C$. The locus $\cM_{g,n}(\ul{k})$ is a subvariety of $\cM_{g,n}$. There is a natural algebraic line bundle over $\cM_{g,n}(\ul{k})$ defined as follows: let $p:\Ccal_{g,n}\ra \cM_{g,n}$ be the universal curve over $\cM_{g,n}$, and $\sigma_1,\dots,\sigma_n$ be the sections of $p$ associated with the marked points. Let $D_i\subset \Ccal_{g,n}$ be the image of $\sigma_i$. Note that $D_i$ is a divisor of $\Ccal_{g,n}$.
Let $\Lcal_i$ denote the line bundle $\Ocal(-D_i)$ on $\Ccal_{g,n}$. Let $K_{\Ccal_{g,n}/\cM_{g,n}}$ be the relative canonical line bundle associated with $p$, and
$$
\Kcal:=K_{\Ccal_{g,n}/\cM_{g,n}}\otimes\Lcal_1^{\otimes k_1}\otimes\dots\otimes\Lcal_n^{\otimes k_n}.
$$
By definition, the restriction of $\Kcal$ to the fiber of $p$ over every point in $\cM_{g,n}(\ul{k})$ is the trivial line bundle. Hence $\Lcal:=p_*(\Kcal_{|p^{-1}(\cM_{g,n}(\ul{k}))})$ is a line bundle over  $\cM_{g,n}(\ul{k})$ which will be called the tautological line bundle.
The complement of the zero section in the total space of $\Lcal$  is the set of tuples $(C,x_1,\dots,x_n,\omg)$, where $(C,x_1,\dots,x_n) \in \cM_{g,n}(\ul{k})$, and $\omg$ is a holomorphic $1$-form on $C$ such that $\mathrm{div}(\omg)=k_1x_1+\dots+k_nx_n$. We denote this set by $\Omega\cM_{g,n}(\ul{k})$. By construction $\Omega\cM_{g,n}(\ul{k})$ is a $\C^*$-bundle over $\cM_{g,n}(\ul{k})$.

There is a natural map $\Fcal: \Omega\cM_{g,n}(\ul{k}) \ra \trate$ which consists in forgetting the numbering of the marked points.
This is actually a finite morphism of algebraic varieties, which is also an orbifold covering between the underlying complex analytic spaces.
Given a point $\xx=(C,\omg) \in \trate$, we will often implicitly endow the set $Z(\omg)$ with a  compatible numbering, which means that we actually consider a point in $\Omega\cM_{g,n}(\ul{k})$ that projects to $\xx$.
That $\Fcal$ is an orbifold covering implies that this choice of numbering can be made consistently in  an orbifold  local chart.


By a {\em linear submanifold} of $\Omega\cM_{g,n}(\ul{k})$ we will mean a subvariety having the property described in Definition~\ref{def:lin:man}.
By extension, we will also call any algebraic variety which admits a finite morphism into $\Omega\cM_{g,n}(\ul{k})$ whose image has the property of Definition~\ref{def:lin:man} a linear submanifold of $\Omega\cM_{g,n}(\ul{k})$.

\subsection{Variation of Hodge structure associated with a stratum}\label{subsec:VMHS}\null\ \\
Let us now fix a stratum $\cS=\trate$. As we mentioned earlier, each stratum $\cS$ is an orbifold, as well as $\Pb\cS$. Actually, there exists a manifold $\Pb\hat\cS$ which is an orbifold covering of $\Pb\cS$.
To see this, we first recall that the forgetful map $\Fcal: \Omega\cM_{g,n}(\ul{k}) \ra \cS$ is an orbifold covering of finite degree. This map induces an orbifold covering $\hat{\Fcal}: \Pb\Omega\cM_{g,n}(\ul{k}) \ra \Pb\cS$. Since the complex line generated by a (non-trivial) Abelian differential is uniquely determined by its divisor, we can identify $\Pb\Omega\cM_{g,n}(\ul{k})$ with $\cM_{g,n}(\ul{k})$. This means that $\cM_{g,n}(\ul{k})$ is an orbifold covering of $\Pb\cS$.

By definition $\cM_{g,n}=\Tcal_{g,n}/\G_{g,n}$, where $\Tcal_{g,n}$ is the Teichm\"uller space of smooth curves of genus $g$ with $n$ marked points, and $\G_{g,n}$ is the corresponding modular group. Orbifold points of $\cM_{g,n}(\ul{k})$ correspond to fixed points of finite subgroups of $\G_{g,n}$. It is well known that $\G_{g,n}$ contains torsion free finite index subgroups. The preimage of $\cM_{g,n}(\ul{k})$ in some orbifold covering of $\cM_{g,n}$ associated with such subgroups is actually a complex manifold. In conclusion, we see that $\Pb\cS$ admits an orbifold covering $\Pb\hat{\cS}$ of finite degree  which is a smooth quasi-projective variety, and over which we have a universal family of $n$-pointed genus $g$ smooth curves.
Passing to a larger covering, one can even assume that there exists a finite group $\Gamma$ acting holomorphically on $\Pb\hat\cS$ such that $\Pb\cS=\Pb\hat\cS/\Gamma$.
Objects over $\Pb\cS$ are defined over $\Pb\hat\cS$ and endowed with an action of $\Gamma$.

Let $\cL$ be the pullback to $\Pb\hat\cS$ of the tautological line bundle over $\Pb\cS$.
Denote by $\hat\cS$ the total space of $\cL$ with the zero section removed.
By construction $\Pb\hat\cS=\hat\cS/\C^*$.
We will denote by $\pi:\hat\cS\rightarrow\Pb\hat\cS$ the natural projection.

Over $\Pb\hat\cS$, the relative homology groups $\HH_1 (C, Z(\omega); \Z)$ assemble in a $\Z$-local system.
We denote by $\cH_{rel}^1$ the dual $\Z$-local system, whose fiber is identified with $\HH^1(C,Z(\omega);\Z)$.
Note that the integration of the form $\omega$ along relative cycles as defined in \textsection\ref{sec:local:coord} can be seen as a holomorphic section $\tau$ over $\hat\cS$ of the holomorphic vector bundle $\pi^{-1}\cH_{rel}^1 \otimes_{\Z} \Ocal_{\hat\cS}$.

Since $\hat{\cS}$ and $\Pb\hat{\cS}$ are locally identified with $\cS$ and $\Pb\cS$ respectively, we will  often identify  elements of $\hat\cS$ (resp. of $\Pb\hat\cS$) with  elements of $\cS$ (resp. $\hat{\cS}$).
For any $(C, [\omega])$ in $\Pb\hat\cS $, denoting by $Z = Z(\omega)$ the zeroes of $\omega$, the relative cohomology group $\HH^1 (C, Z; \Z)$ fits in an exact sequence
\[ 0 \rightarrow \HH^0(C, \Z) \rightarrow  \HH^0(Z, \Z)   \rightarrow \HH^1(C, Z; \Z) \rightarrow \HH^1(C, \Z) \rightarrow 0.\]
Setting $\tilde \HH^0(Z, \Z) := \HH^0(Z, \Z) / \HH^0(C, \Z)$, this yields
\begin{equation}\label{eq:Hodge:mixed:ex:seq}
0 \rightarrow \tilde  \HH^0(Z, \Z)   \rightarrow \HH^1(C, Z; \Z) \rightarrow \HH^1(C, \Z) \rightarrow 0.
\end{equation}
On $\HH^0(C, \C)$ and $\HH^0(Z, \C)$ we have canonical (positive) polarizations defined over $\Z$. We endow $\tilde \HH^0(Z, \C)$ with the quotient polarization, which will be denoted by $h_0$. As for $\HH^1(C, \C)$, it is endowed with the Hermitian pseudo-metric $h_1$ of signature $(g,g)$ defined by
\[h_1(u,v) := \frac\imath2\int_C  u \wedge \bar v=\frac{\imath}{2}\sum_{j=1}^g u(a_j)\overline{v(b_j)} - u(b_j)\overline{v(a_j)}\]
where $(a_1,\dots,a_g,b_1,\dots,b_g)$ is any symplectic basis of $H_1(C,\Z)$.
Observe that the corresponding skew-symmetric form is defined over $\Z$.

Denoting by $F$ the complex vector subspace $\HH^0(C, \Omega^1)$ of $ \HH^1(C, \C)$ formed by cohomology classes of holomorphic $1$-forms, the Hodge filtration is the decreasing filtration $F^{\sbt}$ of $\HH^1(C, \C)$ defined by
\[ F^{2} =  \{0\} \subset F = F^1  \subset F^0 =  \HH^1(C,  \C).\]
If we let $(C, [\omega])$ vary in $\Pb\hat\cS$, then $W:=\tilde  \HH^0(Z, \Z)$ and $\HH^1(C, \Z)$ assemble in $\Z$-local systems $\cW$ and $\cH^1$ that fit in an exact sequence of $\Z$-local systems
\[ 0 \rightarrow \cW   \rightarrow \cH^1_{rel} \rightarrow \cH^1 \rightarrow 0.\]
Remark that with our assumption on $\Pb\hat\cS$, $\cW$ is actually constant. The Hermitian forms $h_0$ and $h_1$ induce a flat (constant) Hermitian metric $h_0$ on $\cW_{\C}$ and a flat Hermitian pseudo-metric $h_1$ on $\cH^1_{\C}$ respectively.  Correspondingly, the $\Z$-local system $\cH^1$ (resp. $\cW$) supports a variation of polarized $\Z$-pure  Hodge structure of weight 1 (resp. of weight 0 associated with the trivial filtration).

The group $ \HH^1(C, Z; \Z) $ is endowed with a graded-polarized $\Z$-mixed Hodge structure, and the exact sequence \eqref{eq:Hodge:mixed:ex:seq} expresses it as an extension of pure polarized $\Z$-Hodge structures.
The weight filtration $W_{\sbt}$ of $\HH^1(C, Z; \Z)$ is defined by
\[ W_{-1} =  \{0\} \subset W = W_0 =  \tilde  \HH^0(Z, \Z)  \subset W_1 =  \HH^1(C, Z;\Z)\]
and the Hodge filtration of $ \HH^1(C, Z; \C)$ is the pullback of the Hodge filtration on $\HH^1(C, \C)$.
This defines on $\cH^1_{rel}$ a variation of graded-polarized $\Z$-mixed Hodge structure.

\subsection{Linear submanifolds revisited}\label{subsec:linsubrev}
Let  $\Pb\cM\subset\Pb\cS$ be a linear submanifold (cf. Definition~\ref{def:lin:man}). 
The preimage $\Pb\hat{\cM}$ of $\Pb\cM$ in $\Pb\hat{\cS}$ is  also a linear submanifold of $\Pb\hat{\cS}$.
The normalization of $\Pb\hat{\cM}$ is then a smooth algebraic variety equipped with an immersive finite generically one-to-one morphism into $\Pb\hat\cS$.
We abusively denote this smooth variety by $\Pb\hat\cM$.
Define  $\hat\cM$ to be the pullback of the tautological $\C^*$-bundle over $\Pb\hat\cS$ to  $\Pb\hat\cM$.
By construction, there is a finite generically one-to-one morphism from $\hat{\cM}$ to $\hat{\cS}$.
Locally on $\hat\cM$ (in the Euclidean topology), the local irreducible components of the image of this map in $\hat{\cS}$ are defined by linear equations in period coordinates.

By definition, given $(C,[\omg]) \in \Pb\hat{\cM} \subset \Pb\hat{\cS}$, there is a linear subspace $V \subset H^1(C,Z(\omg);\C)$ such that a neighborhood of $(C,[\omg])$ in $\Pb\hat{\cM}$ can be identified with an open subset of $\Pb V:=(V\setminus\{0\})/\C^*$ via a period mapping. Define $V_1:=\pp(V)\subset H^1(C,\C)$, and $V_0:=\ker\pp\cap V$. We then have the following exact sequence
\begin{equation}\label{eq:exact:seq:lin:subv}
0 \ra V_0 \ra V \overset{\pp }{\ra} V_1 \ra 0.
\end{equation}
The trivial bundles with fibers $V_0,V,V_1$ over open subsets as above patch together to form three local systems over $\Pb\hat{\cM}$, which will be denoted by $\cV_0,\cV,\cV_1$ respectively.
Note that $\cV_0,\cV,\cV_1$ are sub-local systems of $\cW_{\C}, (\cH^1_{\rm rel})_{\C}, \cH^1_{\C}$ respectively. We have the following exact sequence
\begin{equation}\label{eq:exact:seq:loc:sys:lin:subv}
0 \rightarrow \cV_0   \rightarrow \cV \rightarrow \cV_1 \rightarrow 0.
\end{equation}

By Deligne~\cite{Deligne:TH2}, the $\Q$-local system $(\cH^1_{\Q})_{|\Pb\hat\cM}$ is semi-simple and hence
\begin{Proposition}
The $\C$-local system $\cV_1$ is semi-simple.
\end{Proposition}

\begin{Remark}\label{rk:loc:syst:tangent:bdl}
Let ${\rm pr}: \hat{\cM} \ra \Pb\hat{\cM}$ be the natural projection. By definition,  ${\pr}^{-1}\cV$ is identified with the tangent bundle to $\hat\cM$.
\end{Remark}

We still denote by $h_0$ resp. $h_1$ their restriction to $\cV_0$ resp. $\cV_1$ . As they are flat, $(\det h_0 \otimes |\det h_1|)^\vee$ defines a flat Hermitian form
\footnote{For any finite dimensional $\C$-vector space $W$ endowed with a non degenerate Hermitian form $h$, we define the Hermitian form $\det h$ on $\det W=\bigwedge^{\dim_\C W} W$ by $((\det h)(w_1\wedge\dots\wedge w_r))^2:=\det (h(w_i,w_j))_{1\leq i,j\leq r}$, if $(w_1,\dots,w_r)$ is any basis of $W$}
on the canonical bundle $K_{\hat\cM}\simeq\pr^{-1}(\det \cV)^\vee$.
However, if $h_1$ is degenerate (in restriction to $\cV_1$), this Hermitian form vanishes, something we would like to avoid. This leads us to the following

\begin{Definition}\label{def:polarized}
A {\em polarized} linear submanifold $\Pb\cM$ of $\Pb\trate$ is a linear submanifold such that $\cV_1$ is a subvariation of Hodge structure of $(\cH^1_\C)_{|\Pb\hat\cM}$. In particular, the restriction of $h_1$ to $\cV_1$ is non degenerate.
\end{Definition}

Note that this condition is actually automatically satisfied if we assume that no other local subsystem of $(\cH^1_{\C})_{|\Pb\hat\cM}$ is isomorphic to $\cV_1$. Indeed in that case, $\cV_1$ inherits from $(\cH^1_{\C})_{|\Pb\hat\cM}$ a structure of complex variation of Hodge structures, see~\cite[\textsection1.12-13]{Deligne:TFM}.
Concretely, this means that the following decomposition holds
\[\cV_1 = (F \cap \cV_1) \oplus (\bar F \cap \cV_1).\]

\begin{Remark}\label{rk:choice:of:cover}
Definition~\ref{def:polarized} is independent of the choice of $\hat\cM$.
\end{Remark}

\section{Volume form}\label{sec:volform}

\subsection{Push forward measure}\label{sec:push:measure}
Let $\cM$ be a polarized linear submanifold in $\cS$.
Consider the exact sequence \eqref{eq:exact:seq:lin:subv}.
Recall that we are given a positive Hermitian form $h_0$ on $V_0$ and a non-degenerate Hermitian form $h_1$ on $V_1$. We abusively denote by $h_1$ the (degenerate) Hermitian form induced by $h_1$ on $V$.
The Hermitian forms $h_0$ and $h_1$ define a Hermitian metric $\det(h_0)\otimes|\det(h)_1|$ on $\bigwedge^d_\C V$ as follows:  let $(e_1,\dots,e_d)$ be a $\C$-basis of $V$, where $(e_1,\dots,e_r)$ is a basis of $V_0$. We define
$$
(\det h_0\otimes|\det h_1|(e_1\wedge\dots\wedge e_d))^2:=\det(h_0(e_i,e_j)_{1\leq i,j\leq r})\cdot|\det(h_1(e_\ell,e_m)_{r+1\leq \ell,m\leq d})|
$$
Let $(\det h_0\otimes \det h_1)^\vee$ denote the dual metric on $\bigwedge^d_\C V^*$.
Let $\sigma \in \bigwedge^d_\C V^*$ be an element of norm 1 with respect to this metric.
Then $\left(\frac{\imath}{2}\right)^d\sigma\wedge\bar{\sigma}$ is a real volume form on $V$.
Since  $\hat{\cM}$ is locally identified with $V$, and $h_0,h_1$ are invariant by the monodromies, this volume form gives a well defined volume form $d\vol$ on $\hat{\cM}$.

\medskip

When $\cM=\cS$, that is $V=H^1(C,Z;\C)$, one can define a volume form on $\cS$ using the invariance of the lattice $H^1(C,Z;\Z\oplus\imath\Z)$.
This volume form is commonly known as the {\em Masur-Veech measure} of the stratum $\cS$.
The Masur-Veech measure is particularly relevant for Teichm\"uller dynamics in moduli spaces and billiards in rational polygons.
By construction,  $d\vol$ differs from the Masur-Veech volume form by a constant.
If $\cM$ is a proper linear submanifold of $\cS$ which is locally defined by homogeneous equations with real coefficients, then by the work of Eskin-Mirzakhani~\cite{EM18}, $\cM$ carries a distinguished volume form which also differs from $d\vol$ by a constant.

\medskip

The volume form $d\vol$ induces a measure $\mu$ on $\Pb\hat\cM$ in the following way: let ${\rm pr}:\hat\cM\rightarrow\Pb\hat\cM$ be the projection.
Recall that $\hat\cM$ is identified with the total space of the tautological bundle $\cL$ over $\Pb\hat\cM$ with the zero section removed. On the line bundle $\cL$, we have the  Hodge metric $||.||$ given by $ ||\omg||^2:=\frac{\imath}{2}\int_C\omg\wedge\ol{\omg}$, for every $(C,\omg) \in \hat{\cM}$.
Let $\hat\cM_{\leq 1}$ be the set of $(C,\omg) \subset\hat\cM$ such that $||\omg|| \leq 1$.
The measure $\mu$ on $\Pb\hat\cM$ is defined by the formula $\mu(B)=\vol({\rm pr}^{-1}(B)\cap \hat\cM_{\leq 1})$, where $B$ is any Borel subset of $\Pb\hat\cM$.

\medskip

Let us now give a description of $\mu$ in more concrete terms.
Let $(C,\omega)$ be a point in $\hat{\cM}$, and $(C,[\omg])$ be its projection in $\Pb\hat{\cM}$.
By some period mapping $\phi$, a neighborhood of $(C,\omg)$ in $\hat{\cM}$ is identified with an open subset of a linear subspace $V \subset H^1(C,Z(\omg);\C)$, which is the fiber of $\cV$ over $(C,[\omega])$.
Note that if $v=\phi(C,\omg)$ then $\|\omg\|^2=h_1(v,v) >0$. Thus the image of $\phi$ is contained in the cone $C^+:=\{v\in V, \; h_1(v,v)>0\}$.
Let $\Pb^+(V):= C^+/\C^* \subset\Pb(V)$ and $\proj : C^+ \rightarrow \Pb^+(V) $ the natural projection.
Let also $C^+_{\leq 1} := \{ v \in V \, | \,  0 < h_1(v,v) \leq 1 \}$.
For any subset $A\subset \Pb^+(V)$, let $C_{\leq 1}(A):=\proj^{-1}(A)\cap C^+_{\leq 1}$.

Since the period mapping $\phi$ is equivariant with respect to the $\C^*$-actions on $\hat{\cM}$ and $V$, it induces a biholomorphism $\hat{\phi}$ from a neighborhood $U$ of $(C,[\omg])$ in $\Pb\hat{\cM}$ onto an open subset $\Omega$ of $\Pb^+(V)$.
Moreover, $\phi$ induces an isomorphism of Hermitian line bundles $\cL_{|U}\simeq \Lb_{\Pb^+(V)|\Omega}$, where $\Lb_{\Pb^+(V)}$ is the tautological line bundle over $\Pb^+(V)$ endowed with the Hermitian metric $h_1$ .
In this setting, the restriction of the measure $\mu$ to $U$ is given as follows: for any Borel subset $B\subset U$, we have
\[ \mu(B)=\vol(C_{\leq 1}(\hat{\phi}(B))). \]
From this description, it is not difficult to see that $\mu$ is actually induced by a volume form $d\mu$ on $\Pb\hat\cM$ (see Lemma~\ref{lm:pushed:vol:form} below).

\subsection{Alternate definition of $d\mu$}\label{sec:defvolform}
Let us describe the construction of $d\mu$ from  another point of view.
We denote by $G \subset \GL(V)$ the subgroup consisting in automorphisms that act as the identity on $V_0$ and whose induced automorphism on $V_1$ preserves $h_1$.
The group $G$ fits in an exact sequence $ 0 \rightarrow {\rm Hom}(V_1, V_0) \rightarrow G \rightarrow U(V_1, h_1) \rightarrow  1$. One easily checks that $\Pb^+(V) $ is an orbit for the action of $G$ on $\Pb(V)$.
Since $\hat\cM$ is locally modeled on $(G,C^+)$, objects on $\hat\cM$ and $\Pb\hat\cM$ can be defined as $G$-invariant objects on $C^+$ and  $\Pb^+(V)$.

\medskip

A volume form on a complex manifold $X$  can be viewed as a section of the bundle $K_X\otimes\bar K_X$ where $K_X$ is the canonical bundle of $X$. A Hermitian metric on the canonical bundle of $X$ induces a metric $|\,.\,|$ on volume forms. A complex manifold being orientable, it always admits a non vanishing global volume form $dV$. We will say that the volume form $dV/|dV|$ is the volume form associated with the metric.


Since $\Pb^+(V)$ is an homogeneous manifold for the action of $G$, its tangent bundle, and hence its canonical bundle, is naturally a $G$-equivariant bundle. This is also the case for the restriction to $\Pb^+(V)$ of the tautological line bundle $\Lb$ over $\Pb(V)$.
Note also that $h_1$ induces on $\Lb$ a $G$-invariant Hermitian metric that we still denote by $h_1$.
On $\Pb^+(V)$ we have the Euler exact sequence  of bundles
$$
0\rightarrow \Omega^1_{\Pb^+(V)}\rightarrow V^\vee\otimes \Lb\rightarrow \C \rightarrow 0.
$$
(see \cite[Theorem II.8.13]{Hart}). It follows that the canonical line bundle $K_{\Pb^+(V)}$ is isomorphic to $\det(V)^\vee\otimes \Lb^{\otimes\dim V}$. Thus $(\det h_0 \otimes |\det h_1| )^\vee \otimes h_1^{\otimes \dim V }$ defines a $G$-invariant metric on $K_{\Pb^+(V)}$.

Specifically, let $(e_1,\dots,e_r)$ be an orthonormal basis of $V_0$ with respect to $h_0$, and $(e'_{r+1},\dots,e'_d)$ an orthonormal basis of $V_1$ with respect to $h_1$. For $i=r+1,\dots,d$, let $e_i$ be a vector in $V$ that projects to $e'_i$.
Then $(e_1,\dots,e_d)$ is a basis of $V$, and $e_1\wedge\dots\wedge e_d \in \det V$ has norm 1 with respect to $\det h_0 \otimes |\det h_1|$.
Let $(z_1,\dots,z_d)$ be the coordinates of $V$ in the basis $(e_1,\dots,e_d)$.
Then the associated volume form on $V$ is
$$
d\vol=\left(\frac{\imath}{2}\right)^{d} dz_1\wedge d\bar{z}_1\wedge\dots\wedge dz_{d}\wedge d\bar{z}_{d}.
$$
Let $v$ be a vector in $C^+(V)$.  Since $h_1(v,v)>0$,  there is some $i\in \{r+1,\dots,d\}$ such that then $z_i(v)\not=0$.
We can assume that $z_d(v)\neq 0$.
In a neighborhood of $[v]$ in $\Pb^+(V)$, we have the local coordinates $w:=(w_1,\dots,w_{d-1})=\Bigl(\frac{z_1}{z_d},\dots,\frac{z_{d-1}}{z_d}\Bigr)$. Using the isomorphism provided by the Euler exact sequence, the volume form associated with $(\det h_0 \otimes |\det h_1| )^\vee \otimes h_1^{\otimes \dim V}$ on $\Pb^+(V)$ writes
$$
d\lambda= \left(\frac{\imath}{2}\right)^{d-1}\cdot\frac{1}{h^d(w)} dw_1\wedge d\bar{w}_1\wedge \dots \wedge dw_{d-1}\wedge d\bar{w}_{d-1}
$$
where $h(w):=h_1\bigl(w_1,\dots,w_{d-1},1\bigr)$.
Recall that in this setting the measure $\mu$ is the pushed forward of the volume form $d\vol$ on $C^+_{\leq 1}$ by the natural projection $\mathrm{pr}_{|C^+_{\leq 1}}: C^+_{\leq 1} \ra \Pb^+(V)$.

\begin{Lemma}\label{lm:pushed:vol:form}
The measure $\mu$ is defined by a volume form $d\mu$ on $\Pb^+(V)$, and we have
$$
d\mu=\frac{\pi}{\dim V}d\lambda.
$$
\end{Lemma}
\begin{proof}
Let $B$ be an open neighborhood of $[v]$ in $\Pb^+(V)$. Consider the map
$$
\begin{array}{cccc}
\phi: & [0;2\pi]\times\R_+^*\times B & \ra & V\\
             & (\theta,t,w_1,\dots,w_{d-1}) & \mapsto & te^{\imath\theta}(w_1,\dots,w_{d-1},1)=(z_1,\dots,z_d).
\end{array}
$$
We have $C(B):=\phi([0;2\pi]\times\R^*_+\times B)\subset V\setminus\{0\}$ is the cone over $B$. By definition,
$$
\phi^{-1}(C(B)\cap C^+_{<1})=\biggl\{(\theta,t,w) \in [0;2\pi]\times\R_+^*\times B, \; 0< t < \frac{1}{\sqrt{h(w)}}\biggr\}.
$$
A quick computation gives
$$
\phi^*d\vol=\left(\frac{\imath}{2}\right)^{d-1} t^{2d-1}dt\wedge d\theta\wedge dw_1\wedge d\bar{w}_1\wedge\dots\wedge dw_{d-1}\wedge d\bar{w}_{d-1}.
$$
It follows that
\begin{eqnarray*}
\mu(B)&=& \left(\frac{\imath}{2}\right)^{d-1} \int_0^{2\pi}d\theta \int_B \left(\int_0^{\frac{1}{\sqrt{h(w)}}} 2t^{2d-1}dt\right)  dw_1d\bar{w}_1\dots dw_{d-1}d\bar{w}_{d-1}\\
     &=& \left(\frac{\imath}{2}\right)^{d-1}\cdot\frac{\pi}{d}\int_B \frac{1}{h^{d}(w)}dw_1d\bar{w}_1\dots dw_{d-1}d\bar{w}_{d-1}
\end{eqnarray*}
which implies that $\mu$ is induced by the volume form
$$
d\mu:=\left(\frac{\imath}{2}\right)^{d-1}\cdot\frac{\pi}{dh^d(w)}dw_1dw_1\dots dw_{d-1}d\bar{w}_{d-1}=\frac{\pi}{d}d\lambda.
$$
\end{proof}

\subsection{Pure case}\label{sec:vol:form:pure}
We now turn to the case where the term $V_0$ in \eqref{eq:exact:seq:lin:subv} is trivial.
\begin{Lemma}\label{lm:vol:form:n:Chern:cl}
Assume that $\cM$ is a polarized absolutely rigid linear submanifold of some stratum $\cS$.
  Let $\Theta$ denote the curvature form of the Hodge metric on $\cL$. 
  Let $(p,q)$ be the signature of the restriction of $h_1$ to $V$.
  Then we have the following equality at every point in $\Pb\hat{\cM}$.
\[ d\mu=\frac{(-1)^{p-1}2\pi}{2^dd!}(\imath\Theta)^{d-1},\]
where $d=\dim\cM$.
\end{Lemma}
\begin{proof}
Recall that for polarized absolutely rigid submanifolds,  the map $\pp:V \to V_1$ in \eqref{eq:exact:seq:lin:subv} is an isomorphism.
Thus $h_1$ is a non-degenerate Hermitian form on $V$, which means that $d=p+q$.

As above, let $\Lb$ denote the restriction of the tautological line bundle on $\Pb(V)$ to $\Pb^+(V)$. Over $\Pb^+(V)$, $h_1$ provides us with a Hermitian metric on $\Lb$.
Since locally the Hermitian line bundle $(\cL,||.||)$ over $\Pb\hat{\cM}$ (where $||.||$ is the Hodge metric) is identified with  $(\Lb_{\Pb^+(V)},h_1)$, Lemma~\ref{lm:vol:form:n:Chern:cl} is an immediate consequence of Lemma~\ref{lm:vol:form:chern:curv} here below.
\end{proof}

\begin{Lemma}\label{lm:vol:form:chern:curv}
Let  $d\mu$ be  the measure on $\Pb^+(V)$, which is the push forward of $d\vol$ defined as in \textsection\ref{sec:push:measure}.
Let $\Theta$ be the  curvature form of the Hermitian metric $h_1$ on $\Lb$ (note that $\imath\Theta$ is a real $(1,1)$-form on $\Pb^+(V)$).
Then we have
$$
d\mu=\frac{(-1)^{p-1}2\pi}{2^{p+q}(p+q)!}(\imath\Theta)^{p+q-1}.
$$
\end{Lemma}
\begin{proof}
We can identify $V$ with $\C^{p+q}$ endowed with the Hermitian form
$$
h_1(z,z)=\sum_{i=1}^p |z_i|^2-\sum_{i=1}^q |z_{p+i}|^2,
$$
where $z=(z_1,\dots,z_{p+q})$. In these coordinates, we have
$$
d\vol=|\det(h_1)|=\left(\frac{\imath}{2}\right)^{p+q}dz_1\wedge d\bar{z}_1\dots dz_{p+q}\wedge d\bar{z}_{p+q}.
$$
Observe that $\Ur(p,q)\simeq \Ur(h_1)$ preserves $\mu$ and the curvature form $\Theta$ (since $\Ur(p,q)$ preserves the Hermitian metric $h_1$ on $\Lb$).
It follows that  $\frac{d\mu}{(\imath\Theta)^{p+q-1}}$ is a function  on $\Pb^+(V)$ invariant under the action of $\Ur(p,q)$.
Since $\Ur(p,q)$ acts transitively on $\Pb^+(V)$, $\frac{d\mu}{(\imath\Theta)^{p+q-1}}$ is actually constant.
To evaluate this constant, it suffices to compute the ratio $\frac{d\mu}{(\imath\Theta)^{p+q-1}}$ at the point $[1:0:\dots:0] \in \Pb^+(V)$.

Consider a neighborhood $B$ of $[1:0:\dots:0]$ in $\Pb^+(V)$ that can be identified with a neighborhood of $0\in \C^{p+q-1}$ by the mapping $(w_1,\dots,w_{p+q-1}) \mapsto [1:w_1:\dots:w_{p+q-1}]$.
By the computations in Lemma~\ref{lm:pushed:vol:form}, we have
$$
d\mu(0)=\left(\frac{\imath}{2}\right)^{p+q-1}\cdot\frac{\pi}{(p+q)}dw_1\wedge d\bar{w}_1\wedge\dots\wedge dw_{p+q-1}\wedge d\bar{w}_{p+q-1}.
$$
To compute $\Theta$, we will use the following section of $\Lb$ over $B$
$$
\begin{array}{cccc}
  \sigma: & B & \ra & \C^{p+q} \\
     & w=(w_1,\dots,w_{p+q-1}) & \mapsto & (1,w_1,\dots,w_{p+q-1}) \\
\end{array}.
$$
Set $h(w):=h_1(\sigma(w),\sigma(w))$, we then have
\begin{eqnarray*}
  \Theta(w) &= &  -\partial\bar{\partial}\log(h(w)) \\
    & = &-\frac{\sum_{i=1}^{p-1}dw_i\wedge d\bar{w}_i-\sum_{i=p}^{p+q-1}dw_i\wedge d\bar{w}_i}{h(w)} + \frac{\partial h(w)\wedge\bar{\partial}h(w)}{h^2(w)}
\end{eqnarray*}
Thus
$$
\Theta(0)=-\sum_{i=1}^{p-1}dw_i\wedge d\bar{w}_i+\sum_{i=p}^{p+q-1}dw_i \wedge d\bar{w}_i.
$$
and
$$
\Theta^{p+q-1}(0)=(p+q-1)!(-1)^{p-1}dw_1\wedge d\bar{w}_1\wedge\dots\wedge dw_{p+q-1}\wedge d\bar{w}_{p+q-1}.
$$
Therefore
$$
\frac{d\mu}{(\imath\Theta)^{p+q-1}}= \frac{\displaystyle  \left(\frac{\imath}{2}\right)^{p+q-1}\cdot\frac{\pi}{(p+q)}}{\imath^{p+q-1} (p+q-1)!(-1)^{p-1}} = \frac{(-1)^{p-1}2\pi}{2^{p+q}(p+q)!}.
$$
\end{proof}

\begin{Remark}\label{rk:push:meas:n:curvature}
We remind the reader that our volume form  $d\vol$ is defined only by using $h_1$.
In fact, if $d\hat{\nu}$ is any volume form on $V$ that is proportional to the Lebesgue measure, then the push forward measure on $\Pb^+(V)$, denoted by $d\nu$, is proportional to $\Theta^{p+q-1}$ by the invariance under the action of $\Ur(h_1)=\Ur(p,q)$.
In the cases where $d\hat{\nu}$ is the Masur-Veech volume form on minimal strata of Abelian differentials $\cH(2g-2)$, or on strata of quadratic differentials with only zeros of odd order, the ratio $\frac{d\nu}{(\imath\Theta)^{p+q-1}}$ has been computed in \cite{Sau18} and \cite{CMS19}.
See also \cite[Lem.5.1]{Sau20} for a related calculation.
\end{Remark}

\section{Computation of the volume in terms of characteristic classes}\label{sec:rational:vol}
Our goal now is to show
\begin{Theorem}\label{th:vol:rational}
Let $\cM$ be a {\bf polarized absolutely rigid} linear submanifold of some stratum $\cS$ of translation surfaces.
Then the volume of $\Pb\cM$ with respect to $\mu$  satisfies
$$
\mu(\Pb\cM)\in \Q\pi^{\dim\cM}.
$$
\end{Theorem}
\begin{proof}
Let $\hat{\cM}, \Pb\hat{\cM}$ be as in \textsection\ref{subsec:linsubrev}, and $\cV_0, \cV$ and $\cV_1$ be the local systems on $\Pb\hat\cM$ defined in \textsection\ref{subsec:linsubrev}.
That $\cM$ is absolutely rigid and polarized means that $\cV_0=\{0\}$, hence $\cV$ and $\cV_1$ are isomorphic, and the restriction of  $h_1$ to $\cV_1$ is non-degenerate. Let $d=\dim\cM=\mathrm{rk}(\cV_1)$, and $(p,d-p)$ be the signature of the restriction of $h_1$ to the fibers of $\cV_1$.  As a consequence, by Lemma~\ref{lm:vol:form:n:Chern:cl}, we have
$$
\mu(\Pb\hat{\cM}):=\int_{\Pb\hat{\cM}}d\mu=\frac{(-1)^{p-1}2\pi}{2^dd!}\int_{\Pb\hat{\cM}}(\imath\Theta)^{d-1},
$$
where $\Theta$ is the curvature of the Hodge metric $\|.\|$ on $\cL$. If there is a compact complex manifold $X$ together with a normal crossing divisor $\partial X$ such that
\begin{itemize}
\item $\Pb\hat{\cM}\simeq X\setminus \partial X$,

\item $\cL$ extends to a holomorphic line bundle $\bar{\cL}$ on $X$, and

\item $(\frac{\imath}{2\pi}\Theta)^{d-1}$ is a representative of $c_1^{d-1}(\bar{\cL})$ in the sense of currents,
\end{itemize}
then
$$
\mu(\Pb\hat{\cM})=\frac{(-1)^{p-1}}{2^dd!}(2\pi)^{d}c_1^{d-1}(\bar{\cL}) \in \Q\cdot\pi^d,
$$
since $c_1^{d-1}(\bar{\cL}) \in H^{2(d-1)}(X,\Z)$.
Unfortunately, the existence of such a compact manifold has not been proven in general.
Nevertheless, we have
\begin{Claim}\label{cl:vol:Chern:cls}
There is a compact complex manifold $\hat{Y}$ together with a normal crossing divisor $\partial \hat{Y}$ and a line bundle $\hat{\cL}\ra \hat{Y}$ such that
\begin{itemize}
\item  there is a finite covering $\hat{q}: \hat{Y}\setminus\partial \hat{Y} \ra \Pb\hat{\cM}$,

\item  $\hat{q}^*(\frac{\imath}{2\pi}\Theta)^{d-1}$ is a representative of $c^{d-1}_1(\hat{\cL})$ in the sense of currents.
\end{itemize}
\end{Claim}
\begin{proof}[Proof of the claim]
Recall that over $\Pb\hat{\cM}$, we have a VHS $\{\cH^1, F\}$, where $F$ is the holomorphic subbundle of $\cH^1_{\C}=\cH^1\otimes_{\Z}\C$ whose fiber over a point $(C,\omg)$ is $H^{1,0}(C)$. By definition, the tautological line bundle $\cL$ is a subbundle of $F$. Moreover the Hermitian metric $||.||$ on $\cL$ is precisely the restriction of the Hodge metric of $\cH^1$ to $\cL$.

Let $\ol{X}$ be a smooth compactification of $\Pb\hat{\cM}$ with normal crossing boundary divisor $\partial \ol{X}$. It is a well known fact that the monodromy of $\cH^1$ about each component of $\partial \ol{X}$ is quasi-unipotent. Thus there is a finite covering $q: \ol{Y} \ra \ol{X}$ ramified over $\partial \ol{X}$ such that the pullback of $\cH^1$ to $Y:=\ol{Y}\setminus \partial\ol{Y}$ has unipotent monodromies about $\partial \ol{Y}:=q^{-1}(\partial \ol{X})$ (see \cite[Th.17]{Kaw81} and \cite[Lem.3.5]{Bru18}).  It follows from the work of Deligne~\cite{Del70} and Schmid~\cite{Sch73} that the pullback of the filtration $\{\cH^1_{\C},F\}$ to $Y$ extends canonically to a filtration $\{\ol{\cH}^1_{\C}, \ol{F}\}$ of holomorphic vector bundles over $\ol{Y}$.
However, the line bundle $\cL$ does not necessarily extends to $\ol{Y}$.
To fix this issue, we construct a modification of $\ol{Y}$ as follows: the pullback of $\cL$ to $Y$ provides us with a section $\sigma$ of the projective bundle $\Pb\cH^1_{\C}$ over $Y$. The closure $\ol{Y}'$ of $\sigma(Y)$  in $\Pb\ol{\cH}^1_{\C}$ is a modification of $\ol{Y}$. Let $\partial \ol{Y}'$ be the preimage of $\partial \ol{Y}$ in $\ol{Y}'$. Note that $\partial\ol{Y}'$ is a divisor of $\ol{Y}'$, and $\ol{Y}'\setminus\partial\ol{Y}'\simeq Y$. The variety $\ol{Y}'$ is not necessarily smooth. Let $(\hat{Y},\partial\hat{Y})$ denote the desingularization of the pair $(\ol{Y}',\partial \ol{Y}')$. By construction, (the pullback of) the filtration $\{\ol{\cH}^1_{\C}, \ol{F}\}$ on $\hat{Y}$ is the canonical extension of the VHS $\{\cH^1,F\}$ on $\hat{Y}\setminus\partial\hat{Y}$. Moreover, the pullback $\hat{\cL}$ of the tautological line bundle $\Ocal(-1)_{\Pb\ol{\cH}^1_{\C}}$ to $\hat{Y}$ is clearly a line subbundle of the pullback of $\ol{\cH}^1_{\C}$.
The restriction of $\hat{\cL}$ to $\hat{Y}\setminus\partial\hat{Y}\simeq Y$ is isomorphic to (the pullback of) $\cL$.
We are now  in position to apply \cite[Theorem 5.1 and Remark 5.19]{Kollar87}, which asserts that any power of the curvature form of $(\hat{\cL}_{|Y}, ||.||)$ is a representative of the corresponding  polynomial in $c_1(\hat{\cL})$ in the sense of currents. This completes the proof of the claim.
\end{proof}

By the claim, let $N$ be the degree of $\hat{q}$, we have
$$
\mu(\Pb\hat{\cM})=\frac{1}{N}\cdot\frac{(-1)^{p-1}}{2^d d!}(2\pi)^d\int_{\hat{Y}\setminus\partial\hat{Y}}\left(\frac{\imath}{2\pi}\Theta\right)^{d-1} = \frac{1}{N}\cdot\frac{(-1)^{p-1}}{2^d d!}(2\pi)^d c_1^{d-1}(\hat{\cL})\in \Q\cdot\pi^d.
$$
Since $\Pb\hat{\cM}$ is a finite cover of $\Pb\cM$, the theorem follows.
\end{proof}

\begin{Remark}\label{rk:compare:ext:line:bdl}
Denote by $\Omega\Mcal_{g,n}$ the pullback of the Hodge bundle $\Omega\Mcal_g$ to $\Mcal_{g,n}$.
Let $\Pb\ol{\cM}$ denote the closure of $\Pb\Mcal$ (more precisely, the closure of its pre-image) in $\Pb\Omega\Mcal_{g,n}$.
This closure is called the {\em incidence variety compactification} of $\Pb\cM$ (cf. \cite{BCGGM:abel}).
By construction, the line bundle $\Lcal$ is precisely the restriction of the tautological line bundle $\Ocal(-1)_{\Pb\Omega\Mcal_{g,n}}$ to $\Pb\cM$.
Thus $\Ocal(-1)_{|\Pb\ol{\Mcal}}$ is an obvious extension of $\Lcal$ to $\Pb\ol{\Mcal}$.
There is a surjective birational morphism from the space  $\hat{Y}$ in Claim~\ref{cl:vol:Chern:cls} onto $\Pb\ol{\Mcal}$.
A natural question one may ask is whether $\hat{\Lcal}$ is  isomorphic to the pullback of $\Ocal(-1)$ to $\hat{Y}$.
If this is true, it would simplify the computation of the integral $\int_{\hat{Y}}c_1^{d-1}(\hat{\Lcal})$, which is equal to the self-intersection number of the divisor associated to $\hat{\Lcal}$.
It seems to us that this should be the case, but we do not have a proof of this fact.
\end{Remark}


\section{Proof of Theorem~\ref{th:aff:subman:count}}\label{sec:proof:theorem2}
Let $\Mcal$ now be an arithmetic affine submanifold of dimension $d$ in some stratum $\trate$.
By definition, $\Mcal$ is locally defined by linear equations with rational coefficients in period coordinates.
By a result of \cite{AEM17},  $\Mcal$ is always polarized.

Let $(C,\omega)$ be a surface in $\Mcal$. As usual, denote by $Z(\omega)$ the zero set of $\omega$.
We identify $H^1(C,Z(\omg);\C)$ with $\C^{2g+n-1}$ using a basis of $H_1(C,Z(\omg);\Z)$.
The image of a neighborhood of $(C,\omega)$ in $\Mcal$ is an open subset of a linear subspace $V\subset \C^{2g+n-1}$ which is defined over $\Q$.
It follows that $\Lambda^\Z_V:=V\cap (\Z\oplus\imath\Z)^{2g+n-1}$ is a lattice in $V$.
Note that the square-tiled surfaces in $\cM$ and close to $(C,\omg)$ are mapped to points in $\Lambda^\Z_V$.

There is a unique volume form on $V$ which is proportional to the Lebesgue measure such that the covolume of $\Lambda^\Z_V$ is equal to one.
This volume form gives rise to a well defined volume form on $\Mcal$, which will be called the Masur-Veech volume form and denoted by $d\vol^*$.
The volume form $d\vol^*$ is important to us because of the following folkloric lemma (see~\cite[Prop. 1.6]{EO01}).

\begin{Lemma}\label{lm:MV:vol:count}
Let $\Mcal_{\leq 1}$ denote the set  of $(C,\omega)\in \Mcal$ such that $||\omega||^2\leq 1$. Then we have
$$
\lim_{m\ra \infty}\frac{\#\ST(\Mcal,m)}{m^d}=\vol^*(\Mcal_{\leq 1}).
$$
\end{Lemma}
\begin{proof}[Sketch of proof]
Let $\cM_1$ denote the set of surfaces in $\cM$ with area equal to $1$. Note that $\cM_1$ is an orbifold of real dimension $2d-1$.
Since $\cM_1$ is paracompact, there is a countable family $\{U_i\}_{i\in \N}$ of compact subsets of $\cM_1$ such that
\begin{itemize}
\item[$\bullet$] $\cM_1=\bigcup_{i\in \N}U_i$,

\item[$\bullet$] each $U_i$ is contained in the image of a local chart by some period mapping, and the boundary of $U_i$ is $C^1$-piecewise,

\item[$\bullet$]  $\inter(U_i)\cap\inter(U_j)=\varnothing$ if $i\neq j$.
\end{itemize}
We can identify $U_i$ with a subset of a linear subspace $V\subset \C^{2g+n-1}$ as above.
Let $C(U_i)$ be the infinite cone $\R^*_+\cdot U_i \subset V$.
Let $\ST(U_i,m)$ be the set of square-tiled surfaces of area at most $m$ whose image is contained in $C(U_i)$.
For all $s >0$, denote by $C_s(U_i)$ the set $\{\eta \in C(U_i), \; 0< h_1(\eta,\eta) \leq s\}$.
By definition
$$
\#\ST(U_i,m)=\#C_m(U_i)\cap\Lambda^\Z_V.
$$
Let $\Delta$ be a fundamental domain for the action of $\Lambda^\Z_V$ in $V$.
We can suppose that $\Delta$ contains $0$ in its interior.
Set
 $$
 W(U_i,m)=\bigcup_{\eta\in\Lambda^{\Z}_V\cap C_{m}(U_i)}(\Delta+\eta) \subset V.
 $$
 By definition, $\vol^*(\Delta)=1$, thus $\vol^*(W(U_i,m))=\#\ST(U_i,m)$.
 We now remark that $\frac{1}{\sqrt{m}}\cdot W(U_i,m)$ approaches $C_1(U_i)$ as $m \ra +\infty$. Thus
 $$
 \lim_{m \ra +\infty}\vol^*(\frac{1}{\sqrt{m}}\cdot W(U_i,m))= \vol^*(C_{1}(U_i)).
 $$
 which implies
 $$
 \lim_{m\ra +\infty}\frac{\#\ST(U_i,m)}{(\sqrt{m})^{2d}}=\vol^*(C_1(U_i)).
 $$
 Summing up over the family $\{U_i, i \in \N\}$ we get the desired conclusion.
\end{proof}

The Masur-Veech volume form on $\cM$ is not necessarily equal to the volume form $d\vol$ defined in \textsection\ref{sec:push:measure}. However, we have

\begin{Proposition}~\label{prop:compare:meas:arith}
The ratio $\displaystyle \frac{d\vol^*}{d\vol}$ is a rational constant on $\cM$.
\end{Proposition}
\begin{proof}
By construction, $\frac{d\vol^*}{d\vol}$ is locally constant. Since $\cM$ is irreducible, it follows that $d\vol^*=\alpha d\vol$ where $\alpha$ is constant on $\cM$. It is enough to show that $\alpha$ is rational in some local chart of $\cM$.

Let $(C,\omega)$ be an element of $\cM$. We fix a basis of $H_1(X,Z(\omega),\Z)$ and identify $H^1(X,Z(\omega);\C)$ with $\C^{2g+n-1}$ using this basis.
Let $z=(z_1,\dots, z_{2g+n-1})$ be the coordinates of $\C^{2g+n-1}$.
Up to a renumbering of these coordinates, the Hermitian form $h_1$ is given by
$$
h_1(z,z)=\frac{\imath}{2}\sum_{j=1}^g(z_j\bar{z}_{g+j}-\bar{z}_jz_{g+j}).
$$
A neighborhood of $(C,\omega)$ in $\cM$ is identified with an open subset of a linear subspace $V\subset \C^{2g+n-1}$ defined over $\Q$.
For any sequence $I=(i_1,\dots,i_r)$, where $i_j\in \{1,\dots,2g+n-1\}$, and $i_j\neq i_{j'}$ if $j\neq j'$, define
$$
\begin{array}{cccc}
\phi_I:& \C^{2g+n-1} & \ra & \C^{I}\\
    & (z_1,\dots,z_{2g+n-1}) & \mapsto & (z_{i_1},\dots,z_{i_r}).
\end{array}
$$
Since $\dim_\C V=d$, there is a sequence $I=(i_1,\dots,i_d)$ such that the projection $\phi_I$ restricts to an isomorphism from $V$ onto $\C^{d}$.

Let $w=(w_1,\dots,w_d)$ be the coordinates of $\C^d$. The inverse of the map $\phi_{I|V}: V \ra \C^d$ is an injective linear map $\psi: \C^d \ra \C^{2g+n-1}$, such that $\psi(\C^d)=V$, and $\phi_I\circ\psi =\id_{\C^d}$. Since $V$ is defined over $\Q$, the matrix of $\psi$ in the canonical bases of $\C^d$ and $\C^{2g+n-1}$ has rational entries. This means that if $\psi(w_1,\dots,w_d)=(z_1,\dots,z_{2g+n-1})$ then $z_j$ is a linear function of $w$ with rational coefficients.
As a consequence, the pullback of the Hermitian form $h_1$ to $\C^{d}$ is given by
$$
\psi^*h_1(w,w)=\frac{\imath}{2}\sum_{1\leq i,j\leq d}\gamma_{ij}w_i\bar{w}_j
$$
where $\gamma_{ij}\in \Q$, and $\gamma_{ij}=-\gamma_{ji}$.
By definition, we have
$$
\psi^*d\vol=\left(\frac{\imath}{2}\right)^d\left|\det(\gamma_{ij})\right|dw_1\wedge\bar{w}_1\wedge\dots\wedge dw_d\wedge\bar{w}_d.
$$

Let $\Lambda^\Z_V$ be the intersection $V\cap(\Z\oplus\imath\Z)^{2g+n-1}$. Since $\Lambda^\Z_V$ is a lattice of $V$, $\phi_I(\Lambda^\Z_V)$ is a lattice of $\C^d$. Remark that $\phi_I(\Lambda^\Z_V)$ is a sublattice of $(\Z\oplus\imath\Z)^d$. Let $r$ be the index of $\phi_I(\Lambda^\Z_V)$ in  $(\Z\oplus\imath\Z)^d$. Observe that the covolume of $(\Z\oplus\imath\Z)^d$ is $1$ with respect to the volume form $\left(\frac{\imath}{2}\right)^ddw_1d\bar{w}_1\dots dw_d d\bar{w}_d$. Thus the covolume of $\phi_I(\Lambda^\Z_V)$ is $r$ with respect to this volume form.
Now, by definition, $\Lambda^\Z_V$ has covolume $1$ with respect to the Masur-Veech volume form $d\vol^*$. Thus $\phi_I(\Lambda^\Z_V)$ has covolume form $1$ with respect to $\psi^*d\vol^*$. This means that
$$
\psi^*d\vol^*=\frac{1}{r}\left(\frac{\imath}{2}\right)^ddw_1d\bar{w}_1\dots dw_d d\bar{w}_d.
$$
Since $\det(\gamma_{ij}) \in \Q$, we have $\frac{\psi^*d\vol^*}{\psi^*d\vol} \in \Q$, and the proposition follows.
\end{proof}

\begin{proof}[Proof of Theorem~\ref{th:aff:subman:count}]
Let $\mu^*$ be the push forward of the Masur-Veech volume form by the projection $\proj_{|\cM_{\leq 1}}: \cM_{\leq 1} \ra \Pb\cM$.
By the arguments of Lemma~\ref{lm:pushed:vol:form}, we see that $\mu^*$ is actually given by a volume form $d\mu^*$ on $\Pb\cM$.
Since $\frac{d\vol^*}{d\vol}=\frac{d\mu^*}{d\mu}$, by Lemma~\ref{lm:MV:vol:count}, we get
$$
\lim_{m\ra \infty}\frac{\#\ST(\cM,m)}{m^d}=\vol^*(\cM_{\leq1}):=\mu^*(\Pb\cM)=\alpha\mu(\Pb\cM).
$$
where $\alpha$ is a rational constant by Proposition~\ref{prop:compare:meas:arith}. By assumption, $\cM$ is a absolutely rigid. It follows from a result of \cite{AEM17} that $\cM$ is polarized. Thus Theorem~\ref{th:vol:rational} allows  us to conclude.
\end{proof}

\section{Counting triangulations and quadrangulations}\label{sec:count:tiling}
\subsection{Flat metrics and $k$-differentials}\label{sec:flat:kdiff}
Let us now fix a finite subset $\Sigma=\{s_1,\dots,s_n\} \subset S$. Consider a triangulation $\G$ of $S$ with profile $\kappa$. We can always assume that $\Sigma$ is the set of singular vertices of $\G$.
There is a diffeomorphism from each face of $\G$ onto the equilateral triangle $\vartriangle=(0,1,e^{\frac{\imath\pi}{3}})$ mapping the edges of $\G$ to the sides of $\vartriangle$.
We endow each face of $\G$ with the Euclidean metric on $\Delta$ via such a map.
This metric extends smoothly across the vertices whose valency is equal to $6$.
At each vertex whose valency is not $6$ (that is, a singular vertex),  we get a conical singularity with cone angle $\frac{\pi e}{3}$, where $e$ is the valency of the vertex.

Let $S':=S\setminus\Sigma$. The linear holonomy of the flat metric provides us with a homomorphism $\rho: \pi_1(S',*) \ra \mathrm{SO}(2)$.
The image of $\rho$ is contained in $\Ub_6:=\{e^{k\frac{2\pi\imath}{6}}, \; k=0,\dots,5\}$.
Thus there exists $k\in \{1,2,3,6\}$ such that $\mathrm{Im}(\rho)=\Ub_k$, where $\Ub_k$ is the group of $k$-th roots of unity.
This means that the flat metric on $S$ is induced by a meromorphic $k$-differential $\xi$.
Note that $\xi$ is only determined up to a constant in $\S^1$.
The set of zeros and poles of $\xi$ is precisely the set of singularities of the flat metric, that is $\Sigma$.
Moreover, the orders of zeros and poles of $\xi$ are completely determined by the profile of $\G$.
Namely, the order of $\xi$ at a singular vertex $v$ is given by $k\frac{e_v-6}{6}$.
In particular, $v$ is a pole of $\xi$ if and only if $e_v<6$, and $v$ is a zero of $\xi$ if and only if $e_v>6$.

\medskip

Similarly, if $\G$ is a quadrangulation, then by endowing each face of $\G$ with the Euclidean  metric on the unit square $\square=(0,1,1+\imath,\imath)$, we also get a flat metric on $S$ with conical singularities at $\Sig$. This metric is defined by a meromorphic $k$-differential $\xi$ on $S$ determined up to constant in $\S^1$, where $k\in \{1,2,4\}$. The set of zeros and poles of $\xi$ is equal to $\Sigma$, and the order of $\xi$  at a vertex $v\in \Sigma$ is given by $k\frac{e_v-4}{4}$.

\subsection{Canonical covering}\label{sec:can:cov:kdiff}
Let $\xi$ be a  meromorphic $k$-differential on a compact Riemann surface $X$. In what follows, we will always assume that $\xi$ is not a power of some $k'$-differential with $k'<k$, and that the poles of $\xi$ have order at most $k-1$. It follows from a classical construction  (see for instance \cite{EV92, BCGGM:kdiff, Ngu19}) that there is a cyclic covering $\varpi: \hat{X}\ra X$ of degree $k$, a holomorphic $1$-form $\hat{\omega}$ on $\hat{X}$, and an automorphism $\tau$ of $\hat{X}$ of order $k$ such that
\begin{itemize}
\item[(a)] $\varpi^*\xi=\hat{\omega}^k$,
\item[(b)] $\tau^*\hat{\omega}=\zeta\hat{\omega}$, where $\zeta$ is a primitive $k$-th root of unity,
\item[(c)] $\hat{X}/\langle \tau \rangle \simeq X$.
\end{itemize}
Let $Z$ denote the set of zeros and poles of $\xi$, and $\hat{Z}=\varpi^{-1}(Z)$. Then, by construction, $\hat{Z}$ contains the zero set of $\omega$ and all the branched points of $\varpi$.
Moreover, the genus $\hat{g}$ of $\hat{X}$, the cardinality $\hat{n}$ of $\hat{Z}$, and the action of $\tau$ on $\hat{Z}$ are completely determined by the orders of the zeros and poles of $\xi$.

Note that for any $k$-th root of unity $\lambda$, the triple $(\hX, \lambda\homg,\langle \tau\rangle)$ also satisfies the conditions above. However, since there exists $\ell \in \{0,\dots,k-1\}$ such that $\tau^\ell\cdot(\hX,\homg,\langle \tau\rangle)=(\hX, \lambda\homg,\langle \tau\rangle)$, all these triples are isomorphic. It can also be shown that any two triples satisfying the conditions (a), (b), (c) above are isomorphic.
For this reason, we can call the triple $(\hX,\homg,\langle\tau\rangle)$ the canonical covering of $(X,\xi)$.


The following proposition characterizes the set of $k$-differentials arising from triangulations (resp. quadrangulations) on $S$.
\begin{Proposition}\label{prop:reg:tiled:surface}
Let $(X,\xi)$ and $(\hX,\homg,\langle \tau \rangle)$ be as above. Then
\begin{itemize}
\item[(i)] The $k$-differential $\xi$ arises from a triangulation of $S$ if and only if up to  multiplication by some constant in $\S^1$, $\homg$ satisfies
\begin{equation}\label{eq:cond:kdiff:triang}
\int_c\hat{\omega} \in \Z\oplus\Z e^{\frac{\imath\pi}{3}}, \; \; \forall c\in H_1(\hat{X},\hat{Z},\Z).
\end{equation}
Moreover, the triangulation is uniquely determined by the triple $(\hat{X},\hat{\omega},\langle\tau\rangle)$.

\item[(ii)] The $k$-differential $\xi$ arises from a quadrangulation of $S$ if and only if up to  multiplication by some constant in $\S^1$, $\homg$ satisfies
\begin{equation}\label{eq:cond:kdiff:quad}
\int_c\hat{\omega} \in \Z\oplus\Z\imath, \; \; \forall c\in H_1(\hat{X},\hat{Z},\Z).
\end{equation}
Moreover, the quadrangulation is uniquely determined by the triple $(\hat{X},\hat{\omega},\langle\tau\rangle)$.
\end{itemize}
\end{Proposition}
\begin{proof}
Let us first suppose that the triple $(\hat{X},\hat{\omega},\langle\tau\rangle)$ is obtained from a triangulation on $S$. By construction, we have a triangulation of $\hat{X}$ by equilateral triangles of unit side that is the pullback of the triangulation on $S$.   Since any cycle $c$ in $H_1(\hat{X},\hat{Z};\Z)$ can be represented by a path composed by some edges of this triangulation, up to rotation, we have $\int_c\homg \in \Z\oplus\Z e^{\frac{\imath\pi}{3}}$.

Assume now that $(\hat{X},\hat{\omega})$ satisfies \eqref{eq:cond:kdiff:triang}. Fix a point $\hx_0 \in \hZ$, and define a map $\varphi: \hat{X} \ra \C/(\Z\oplus\Z e^{\frac{\imath\pi}{3}})$ as follows
 $$
 \varphi(x)=\int_{\hat{x}_0}^x \hat{\omega} \mod (\Z\oplus\Z e^{\frac{\imath\pi}{3}})
 $$
where the integral is taken along any path from $\hat{x}_0$ to $x$.
It is straightforward to check that $\varphi$ is a ramified covering with all the branched points contained in $\hat{Z}$, and satisfies $\varphi(\hat{Z})=\{0\}$.
The torus $\C/(\Z\oplus\Z e^{\frac{\imath\pi}{3}})$ admits a triangulation $\G_0$ composed by $2$ equilateral triangles with $0$ being the unique vertex.
The pullback of $\G_0$ by $\varphi$ gives a triangulation of $\hat{X}$.
Since $\tau^*\hat{\omega}=\zeta\hat{\omega}$, where $\zeta \in \Ub_6$,  it follows that there exists an automorphism $j$ of the torus $\C/(\Z\oplus\Z e^{\frac{\imath\pi}{3}})$ which fixes $0$ and satisfies $\varphi\circ\tau=j\circ\varphi$. Note that $j$ preserves the triangulation $\G_0$.
Therefore, $\tau$ preserves the triangulation of $(\hat{X},\hat{\omega})$, which means that
this triangulation descends to a triangulation of $\hat{X}/\langle \tau \rangle\simeq S$.
It is clear that each triple $(\hat{X},\hat{\omega},\langle \tau \rangle)$ provides us with a unique triangulation of $S$ up to homeomorphisms.

The case where $(\hat{X},\hat{\omega},\langle \tau \rangle)$ arises from a quadrangulation of $S$ follows from similar arguments.
\end{proof}
\subsection{Moduli space of $k$-differentials}\label{sec:mod:sp:kdiff}
Let $\kappa:=(\kappa_1,\dots,\kappa_n)$ be an admissible profile for triangulations  (resp. quadrangulations) of $S$. Fix a positive integer $k$ such that $k|6$ and $\frac{6}{k} | \gcd(\kappa_1,\dots,\kappa_n)$ (resp. $k | 4$ and $\frac{4}{k} | \gcd(\kappa_1,\dots,\kappa_n)$). Let $k_i=k\cdot\frac{\kappa_i}{6}$ if $\kappa$ is a profile for triangulations, and $k_i=k\cdot\frac{\kappa_i}{4}$ if  $\kappa$ is a profile  for quadrangulations. In all cases, set $\ul{k}=(k_1,\dots,k_n)$.
Note that we have $k_i\in \Z_{>-k}$, and $k_1+\dots+k_n=k\cdot(2g-2)$.

Let $\Skdiff$ denote the space of pairs $(X,\xi)$ where $X$ is a Riemann surface of genus $g$, and $\xi$ is a meromorphic $k$-differential on $X$ whose zeros and poles have orders prescribed  by $\ul{k}$.
Since every $k$-differential is uniquely determined by its canonical covering, we can identify $\Skdiff$ with the space of triples $(\hX,\homg,\langle\tau\rangle)$ satisfying the conditions (a), (b), (c) in \textsection\ref{sec:can:cov:kdiff}.
Denote by $\pSkdiff$ the projectivization of $\Skdiff$, that is $\pSkdiff=\Skdiff/\C^*$.

Given $(X,\xi)\in \Skdiff$, denote by $Z$ the set of zeros and poles of $\xi$. Let $\hZ$ be the inverse image of $Z$ in $\hX$. By construction, all the zeros of $\homg$ are contained in $\hZ$.
However, some of the points in $\hZ$ may not be zeros of $\homg$. We will consider those points as zeros of order $0$ of $\homg$, and subsequently $\hZ$ as the zero set of $\homg$.
Let $\ul{\hk}:=(\hk_1,\dots,\hk_{\hn})$ be the sequence of non-negative integers recording the orders of the zeros in $\hZ$.
Recall that $\Sabel$ is the stratum of Abelian differentials consisting of triples $(\hX,\homg,\hZ)$, where $\hX$ is a Riemann surfaces of genus $\hg$, $\homg$ an Abelian differential on $\hX$, and $\hZ=\{\hx_1,\dots,\hx_{\hn}\}$ is a finite subset  of $\hX$ such that $\mathrm{div}(\homg)=\hk_1\hx_1+\dots+\hk_{\hn}\hx_{\hn}$.

\begin{Proposition}\label{prop:kdiff:ln:man}
The stratum $\Skdiff$ is a finite cover of a  polarized linear submanifold of $\Sabel$. Moreover, if $k_i\not\in k\cdot\Z$, for all $i=1,\dots,n$, then this linear submanifold is absolutely rigid.
\end{Proposition}
\begin{proof}
Let $(\hX,\homg,\langle\tau\rangle)$ be an element of $\Skdiff$, and $\hZ$ be the zero set of $\homg$. Let $\psi_1: \Skdiff \ra \Sabel$ be the map which sends the triple $(\hX,\homg,\langle \tau \rangle)$ to the triple $(\hX,\homg,\hZ)$, that is we forget about the automorphism $\tau$. This map is actually a finite morphism of algebraic varieties. In particular $\psi_1(\Skdiff)$ is an algebraic subvariety of $\Sabel$.
Recall that a neighborhood of $(\hX,\homg,\hZ)$ in $\Sabel$ can be identified with an open subset of $H^1(\hX,\hZ,\C)$.
There is a neighborhood $\Ucal$ of $(\hX,\homg,\langle\tau\rangle)$ in $\Skdiff$ such that $\psi_1(\Ucal)$ is an open subset of $V_\zeta \subset H^1(\hX,\hZ,\C)$, where $V_\zeta$ is the eigenspace associated with the eigenvalue $\zeta$ of the action of $\tau$ on $H^1(\hX,\hZ,\C)$ (see for instance \cite{BCGGM:kdiff, Ngu19}). Thus $\psi_1(\Skdiff)$ is a linear submanifold of $\Sabel$.

Let $\pp: H^1(\hX,\hZ,\C) \ra H^1(\hX,\C)$ be the natural projection. By \cite[Prop. 5.1]{Ngu19}, the restriction of the intersection form on $H^1(\hX,\C)$ to $\pp(V_\zeta)$ is non-degenerate. Thus $\psi_1(\Skdiff)$ is polarized. Finally, by \cite[Prop.4.1]{Ngu19}, $\dim \ker\pp\cap V_\zeta = \#\{k_i\in k\cdot\Z, \; i=1,\dots,n\}$. Thus if $k_i\not\in k\cdot\Z$ for all $i=1,\dots,n$, then $\pp(V_\zeta)$  is absolutely rigid.
\end{proof}

\subsection{Counting in moduli spaces}\label{sec:count:in:mod:sp}
Let $\Tkdiff$ (resp. $\Qkdiff$) denote the set of $\hxx=(\hX,\homg,\langle\tau\rangle)\in \Skdiff$ which satisfy  the condition \eqref{eq:cond:kdiff:triang} (resp. the condition \eqref{eq:cond:kdiff:quad}) of Proposition~\ref{prop:reg:tiled:surface}.

\begin{Lemma}\label{lm:forms:same:triang}
Let $\hxx=(\hX,\homg,\langle \tau \rangle)$ and $\hxx'=(\hX',\homg',\langle \tau' \rangle)$ be two elements of $\Tkdiff$ (resp. of $\Qkdiff$).
Then $\hxx$ and $\hxx'$ give the same triangulation (resp. quadrangulation) of $S$ if and only if there exists $\alpha\in \Ub_6$ (resp. $\alpha\in \Ub_{4}$) such that $\hxx'\simeq\alpha\cdot\hxx$.
\end{Lemma}
\begin{proof}
We only give the proof for the case $\hxx,\hxx'\in\Tkdiff$.
Assume that $\hxx' \simeq\alpha\cdot\hxx$, where $\alpha^{6}=1$.
This means that there is an isomorphism $f: \hX \ra \hX'$ such that $f^*\homg'=\alpha\homg$ and $\langle \tau\rangle=\langle f^{-1}\circ\tau'\circ f \rangle$.  Since  both $\homg$ and $\homg'$ satisfy \eqref{eq:cond:kdiff:triang},  $f^*\homg'$ and $\homg$ give two triangulations of $\hX$ by unit equilateral triangles (see Proposition~\ref{prop:reg:tiled:surface}). Since $\hf^*\homg'=\alpha\homg$, with $\alpha\in \Ub_6$, the two triangulations coincide. Thus, they induce the same triangulation of $S$. The proof of the converse is left to the reader.
\end{proof}

For each $k \in \{1,2,3,6\}$ such that $\frac{6}{k} | \gcd(\kappa_1,\dots,\kappa_n)$ (resp. each $k\in \{1,2,4\}$ such that $\frac{4}{k} | \gcd(\kappa_1,\dots,\kappa_n)$), denote by $\TT^{(k)}(\kappa)$  (resp. by $\QQ^{(k)}(\kappa)$) the set of triangulations (resp. quadrangulations) such that the corresponding flat surface is induced by a $k$-differential.  We then have a partition of the set $\TT(\kappa)$ into the disjoint union  $\sqcup_{k}\TT^{(k)}(\kappa)$, and a partition of $\QQ(\kappa)$ into the disjoint union  $\sqcup_{k}\QQ^{(k)}(\kappa)$.

For any $m\in \Z_{>0}$, denote by $\Tkdiffm$ the set of elements of $\Tkdiff$ whose canonical triangulation is composed by at most $km$ triangles.
Since the area of an equilateral triangle with unit side is  $\frac{\sqrt{3}}{4}$,  $\hxx \in \Tkdiffm$ if and only if $\hxx$ satisfies \eqref{eq:cond:kdiff:triang}, and
$$
||\homg||^2=\frac{\imath}{2}\int_{\hX}\homg\wedge\ol{\homg} \leq \frac{\sqrt{3}}{4}\cdot km.
$$
By definition, the triangulations of $S$ that are induced by elements of $\Tkdiffm$ have at most $k$ triangles.

Similarly, denote by $\Qkdiffm$ the set of $\hxx\in \Qkdiff$ such that the canonical quadrangulation of $\hxx$ has at most $km$ squares, or equivalently $||\homg||^2 \leq km$. The quadrangulations of $S$ induced by elements of $\Qkdiffm$ have at most $m$ squares.

Let $\TT^{(k)}(\kappa,m)$ (resp. $\QQ^{(k)}(\kappa,m)$) denote the set of triangulations in $\TT^{(k)}(\kappa)$ (resp. quadrangulations in $\QQ^{(k)}(\kappa)$) that are composed by at most $m$ triangles (resp. $m$ squares). By definition, every element of $\Tkdiffm$ (resp. of $\Qkdiffm$) gives us  an element of $\TT^{(k)}(\kappa,m)$ (resp. and element of $\QQ^{(k)}(\kappa,m)$). Notice however that several elements of $\Tkdiffm$ (resp. $\Qkdiffm$) may give the same  element of $\TT^{(k)}(\kappa,m)$ (resp. $\QQ^{(k)}(\kappa,m)$.

\begin{Proposition}\label{prop:triang:n:lat:pts}
For any $k \in \{1,2,3,6\}$ such that $\frac{6}{k} | \gcd(\kappa_1,\dots,\kappa_n)$, we have
\begin{equation}\label{eq:triang:n:lat:pts}
\lim_{m\ra \infty}\frac{\#\Tkdiffm}{\#\TT^{(k)}(\kappa,m)}=\frac{6}{k}.
\end{equation}

For any $k\in \{1,2,4\}$ such that $\frac{4}{k} | \gcd(\kappa_1,\dots,\kappa_n)$, we have
\begin{equation}\label{eq:quad:n:lat:pts}
\lim_{m\ra \infty}\frac{\#\Qkdiffm}{\#\QQ^{(k)}(\kappa,m)}=\frac{4}{k}.
\end{equation}
\end{Proposition}

\begin{proof}
By Lemma~\ref{lm:forms:same:triang}, we have a bijection between $\TT^{(k)}(\kappa,m)$ and the set of $\Ub_6$-orbits in $\Tkdiffm$.
By construction, the stabilizer of a point $\hxx\in \Skdiff$ for the $\Ub_6$ action contains $\Ub_k$. Thus generically, an $\Ub_6$-orbit contains $\frac{6}{k}$ elements. There may exist $\hxx\in \Skdiff$ such that $\Ub_6\cdot\hxx$ contains less than $\frac{6}{k}$ elements, in which case, $\hxx$ is an orbifold point of $\Skdiff$. Since the set of orbifold points of $\Skdiff$ is a (finite) union of proper subvarieties, the number of elements of $\Tkdiffm$ that are orbifold points is negligible compared to $\#\Tkdiffm$ as $m\lra +\infty$. Therefore, we have
$$
\lim_{m\ra \infty}\frac{\#\Tkdiffm}{\#\TT^{(k)}(\kappa,m)}=\frac{6}{k}.
$$
The proof of \eqref{eq:quad:n:lat:pts} follows the same lines.
\end{proof}

\subsection{Comparison with Masur-Veech measure}\label{sec:compare:MV:vol}
By Proposition~\ref{prop:triang:n:lat:pts}, to get the asymptotics of $\TT^{(k)}(\kappa,m)$ and of $\QQ^{(k)}(\kappa,m)$, it suffices to compute the asymptotics of the points in $\Tkdiffm$ and in $\Qkdiffm$ respectively. For this purpose, we introduce some volume forms on $\Skdiff$ as follows: let $\Lambda_3:=(\Z\oplus\Z e^{\frac{2\pi\imath}{3}})^{2\hg+\hn-1}$ and $\Lambda_4:=(\Z\oplus\Z\imath)^{2\hg+\hn-1}$.
Given a point $(\hX,\homg,\langle\tau\rangle) \in \Skdiff$, using a basis of $H_1(\hX,\hZ,\Z)$, we identify $H^1(\hX,\hZ,\C)$ with $\C^{2\hg+\hn-1}$.
Let $V_\zeta \subset \C^{2\hg+\hn-1}$ be the eigenspace associated with the eigenvalue $\zeta$ of the action of $\tau$ on $H^1(\hX,\hZ,\C)$.
A neighborhood of $(\hX,\homg,\langle\tau\rangle)$ in  $\Skdiff$ is then identified with an open subset of $V_\zeta$.

Note that $V_\zeta$ is defined over $\Q(\zeta)$. Thus if $\zeta \in \Ub_6$, then $\Lambda_3(\zeta):=V_\zeta\cap\Lambda_3$ is a lattice of $V_\zeta$, and if $\zeta\in \Ub_4$ then $\Lambda_4(\zeta):=V_\zeta\cap\Lambda_4$ is a lattice of $V_\zeta$.
There is a unique volume form $d\vol_3^*$ (resp. $d\vol_4^*$) on $V_\zeta$ which is proportional to the Lebesgue measure such that the lattice $\Lambda_3(\zeta)$ (resp. $\Lambda_4(\zeta)$) has covolume $1$.
Recall  that the coordinate changes on $\Skdiff$  consist in changing the basis of $H_1(\hX,\hZ,\Z)$, which correspond to applying some matrices in $\GL(2\hg+\hn-1,\Z)$. Note that such matrices preserve the lattices $\Lambda_3,\Lambda_4$. If $V'_\zeta=A(V_\zeta)$ for some $A\in \GL(2\hg+\hn-1,\Z)$, then $\Lambda'_3(\zeta):=V'_\zeta\cap\Lambda_3=A(\Lambda_3(\zeta))$, and $\Lambda'_4(\zeta):=V'_\zeta\cap\Lambda_4=A(\Lambda_4(\zeta))$.
This implies that $\vol_3^*$ and $\vol_4^*$ give rise to two  well defined volume forms on $\Skdiff$, which will be called the {\em Masur-Veech volume forms}. By the same argument as in Lemma~\ref{lm:MV:vol:count}, we get
\begin{Lemma}\label{lm:MV:asymp:triang} Let $\Skdiffi$ denote the set of $(\hX,\homg,\langle\tau\rangle) \in \Skdiff$ such that $||\homg|| <1$. Then

\begin{eqnarray}
\label{eqn:asymp:n:vol:tria} \lim_{m\ra +\infty}\frac{\Tkdiffm}{(\frac{\sqrt{3}}{4}km)^d} & = & \vol_3^*(\Skdiffi) \\
\label{eq:asymp:n:vol:quad} \lim_{m\ra +\infty}\frac{\Qkdiffm}{(km)^d} & = & \vol_4^*(\Skdiffi),
\end{eqnarray}
where $d=\dim_\C\Skdiff$.
\end{Lemma}

The Masur-Veech volume forms induce the volume forms $d\mu_3^*$ and $d\mu_4^*$ on $\pSkdiff$ in the same manner as $d\mu$ is induced from $\vol$.
By definition, we have $\vol_3^*(\Skdiffi)=\mu^*_3(\pSkdiff)$, and $\vol_4^*(\Skdiffi)=\mu^*_4(\pSkdiff)$.
From their construction, there are constants $\lambda_s \in \R^*, \, s\in \{3,4\}$,   such that
$$
\frac{d\vol_s^*}{d\vol}=\frac{d\mu_s^*}{d\mu}=\lambda_s.
$$
The following proposition follows from similar arguments as in Proposition~\ref{prop:compare:meas:arith} (see also \cite[\textsection 5.2]{Ngu19}).
\begin{Proposition}\label{prop:ratio:vol:forms}
Assume that  $\ker\pp\cap V_\zeta=\{0\}$, where $\pp: H^1(\hX,\hZ;\C) \ra  H^1(\hX,\C)$ is the natural projection. Then we have
$$
\frac{d\vol^*}{d\vol}=\frac{d\mu^*}{d\mu} \in \Q.
$$
\end{Proposition}

\subsection{Proof of Theorem~\ref{th:asymp:count}}
\begin{proof}
We will only give the proof for $\TT(\kappa)$ as the proof for $\QQ(\kappa)$ is exactly the same.
Combining Proposition~\ref{prop:triang:n:lat:pts}, Lemma~\ref{lm:MV:asymp:triang}, and the fact that $\vol^*_3(\Skdiffi)=\mu^*_3(\pSkdiff)$, we get
$$
\lim_{m\ra +\infty}\frac{\#\TT^{(k)}(\kappa,m)}{(\frac{\sqrt{3}}{4}km)^d} \in \Q\cdot\mu_3^*(\pSkdiff).
$$
Now the hypothesis on $\kappa$ implies that for all $i\in \{1,\dots,n\}$, the order $k_i$ of the $k$-differentials in $\Skdiff$ is not a multiple of $k$. This implies in particular that $k\neq 1$, that is $\Skdiff$ is not a stratum of Abelian differentials. Therefore (see \cite{BCGGM:kdiff, Ngu19})
$$
d=\dim_\C\Skdiff=2g+n-2.
$$
By Proposition~\ref{prop:kdiff:ln:man} we know that $\Skdiff$ is a polarized absolutely rigid linear submanifold of the stratum $\Sabel$.
Therefore, from Proposition~\ref{prop:ratio:vol:forms} we get
$$
\lim_{m\ra +\infty}\frac{\#\TT^{(k)}(\kappa,m)}{(\frac{\sqrt{3}}{4}km)^{2g+n-2}} \in \Q\cdot\mu(\pSkdiff).
$$
By Theorem~\ref{th:vol:rational}, we know that $\mu(\pSkdiff) \in \Q\cdot\pi^{2g+n-2}$. Thus
$$
\lim_{m\ra +\infty}\frac{\#\TT^{(k)}(\kappa,m)}{m^{2g+n-2}} \in \Q\cdot(\sqrt{3}\pi)^{2g+n-2}.
$$
Since
$$
\#\TT(\kappa,m)=\sum_{k\in\{2,3,6\}} \#\TT^{(k)}(\kappa,m)
$$
the theorem follows.
\end{proof}



\end{document}